\documentclass[a4paper,final,oneside,notitlepage,12pt]{article}

\usepackage{amssymb}
\usepackage{amsthm}
\usepackage{amsmath}
\usepackage[all]{xy}
\usepackage{graphicx}
\usepackage{mathrsfs}
\usepackage[bottom]{footmisc}
\usepackage[margin=2cm,font=normalsize,labelfont=bf]{caption}
\usepackage[normalem]{ulem}
\usepackage{verbatim}
\usepackage{amstext}

\makeatletter

\gdef\@ntitle{\@title}
\def\subtitle#1{\gdef\@subtitle{#1}}
\def\@subtitle{}
\def\adress#1{\gdef\@adress{#1}}
\def\@adress{}
\def\preprint#1{\gdef\@preprint{#1}}
\def\@preprint{}
\def\keywords#1{\gdef\@keywords{#1}}
\def\@keywords{}
\def\email#1{\gdef\@email{#1}}
\def\@email{}

\def\refname{References}

\newlength{\myparskip}
\newlength{\myproofparskip}
\def\href#1#2{#2}

\def\kohyp{
  \usepackage{hyperref}
  \hypersetup{
    linktocpage = true,
    pdftitle = {\@title},
    pdfauthor = {\@author},
    pdfkeywords = {\@keywords},    
    bookmarksopen = true,
    bookmarksopenlevel = 1
  }}  
\def\showkeywords{\begin{flushleft}\footnotesize\textbf{Keywords}: \@keywords.\end{flushleft}}

\newcounter{mythm}[subsection]

\def\setsecnumdepth#1{
  \setcounter{secnumdepth}{#1}
  \setcounter{mythm}{0}
  \ifnum \c@secnumdepth >0
    \ifnum \c@secnumdepth >1
      \def\themythm{\thesubsection.\arabic{mythm}}
      \numberwithin{equation}{subsection}
      \renewcommand\theequation{\thesubsection.\arabic{equation}}
    \else
      \def\themythm{\thesection.\arabic{mythm}}
      \numberwithin{equation}{section}
      \renewcommand\theequation{\thesection.\arabic{equation}}
    \fi
  \else
    \def\themythm{\arabic{mythm}}
  \fi}

\newtheorem{theorem}[mythm]{Theorem}
\newtheorem{definition}[mythm]{Definition}
\newtheorem{rem}[mythm]{Remark}
\newtheorem{corollary}[mythm]{Corollary}
\newtheorem{exa}[mythm]{Example}
\newtheorem{proposition}[mythm]{Proposition}
\newtheorem{lemma}[mythm]{Lemma}

\newenvironment{remark}{\begin{rem}\normalfont\setlength{\parskip}{\myproofparskip}}{\setlength{\parskip}{\myparskip}\end{rem}}
\newenvironment{example}{\begin{exa}\normalfont\setlength{\parskip}{\myproofparskip}}{\setlength{\parskip}{\myparskip}\end{exa}}

\renewenvironment{proof}{Proof.\setlength{\parskip}{\myproofparskip}}{\hfill{$\square$}\\\setlength{\parskip}{\myparskip}}

\def\Z {\mathbb{Z}}

\def\R {\mathbb{R}}

\def\im{\mathrm{i}}
\def\id{\mathrm{id}}

\def\trivlin{\mathbf{I}}
\def\quot#1{``#1''}

\renewcommand{\varepsilon}{\epsilon}
\renewcommand{\emph}[1]{\def\reserved@a{it}\ifx\f@shape\reserved@a\uline{#1}\else\textit{#1}\fi}

\newcommand\erf[1]{(\ref{#1})}

\newlength{\myl}
\newcommand\sheaf[1]{\unitlength 0.1mm
  \settowidth{\myl}{$#1$}
  \addtolength{\myl}{-0.8mm}
  \begin{picture}(0,0)(0,0)
  \put(3,6){\text{\uline{\hspace{\myl}}}}
  \end{picture}#1\hspace{-0.15mm}}
\newcommand{\ueinssheaf}{\sheaf\ueins}

\newcommand{\ueins}{{\mathrm{U}}(1)}

\newcommand{\su}[1]{{\mathrm{SU}}(#1)}

\newcommand{\so}[1]{{\mathrm{SO}}(#1)}

\def\bun#1#2{\buntech#1{}(#2)}

\def\tocsection#1{\section*{#1}\addcontentsline{toc}{section}{#1}}

\def\mytitle{}
\def\zmptitle{
  \begin{tabular}{cc}
    \begin{minipage}[c]{0.4\textwidth}
      \begin{flushleft}
        \includegraphics[width=110pt]{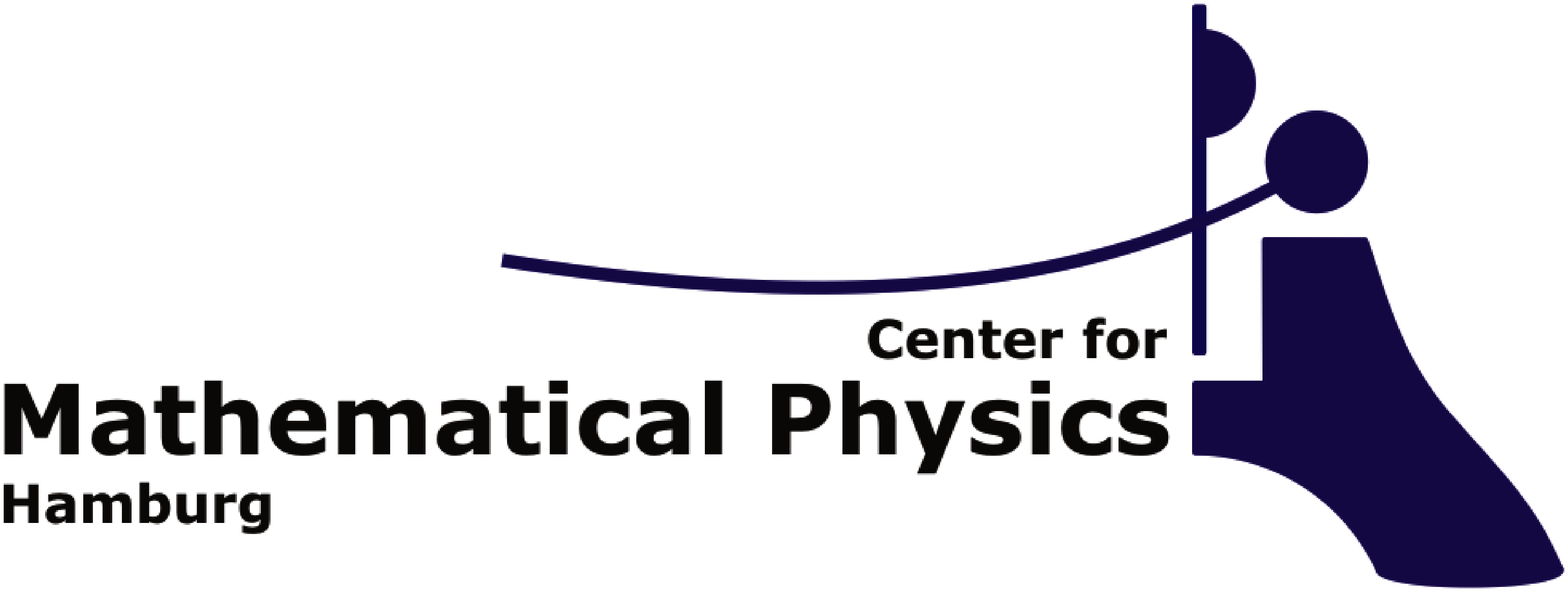}
      \end{flushleft}  
    \end{minipage}&
    \begin{minipage}[c]{0.55\textwidth}
      \begin{flushright}
      {\small\sf\@preprint}
      \end{flushright}
    \end{minipage}
  \end{tabular}
  \vskip 2cm}

\def\maketitle{
  \setlength{\parskip}{\myparskip}
  \newpage
  \noindent
  \mytitle
  \begin{center}
    \LARGE\@ntitle\\
    \if!\@subtitle!\else \smallskip\LARGE\@subtitle\\\fi
    \bigskip
    \if!\@author!\else   \bigskip\large\@author\\\fi
    \if!\@adress!\else   \bigskip\normalsize\@adress\\\fi
    \if!\@email!\else    \bigskip\normalsize\textit{\@email}\\\fi
  \end{center}
  \vskip 2cm\thispagestyle{empty}}

\def\etalchar#1{$^{#1}$}
\def\kobib#1{

}

\def\showcomments{ -- Comments suppressed}

\newif\if@fewtab\@fewtabtrue{
  \count255=\time\divide\count255 by 60
  \xdef\hourmin{\number\count255}
  \multiply\count255 by-60\advance\count255 by\time
  \xdef\hourmin{\hourmin:\ifnum\count255<10 0\fi\the\count255}}
\def\ps@draft{
  \let\@mkboth\@gobbletwo
  \def\@oddfoot{
    \hbox to 7 cm{\tiny \versionno\hfil}
    \hskip -7cm\hfil\rm\thepage\hfil{\tiny\draftdate}}
  \def\@oddhead{}
  \def\@evenhead{}
  \let\@evenfoot\@oddfoot}
\def\draftdate{\number\month/\number\day/\number\year\ \ \ \hourmin }
\newcommand\version[1]{
  \typeout{}\typeout{#1}\typeout{}
  \vskip-1.7cm \centerline{\fbox{{\normalsize\tt DRAFT -- #1 -- 
  \draftdate\showcomments}}} \vskip0.92cm}
\def\draft#1{
  \def\versionno{#1}
  \pagestyle{draft}\thispagestyle{draft}
  \gdef\@ntitle{\version\versionno \@title}
  \global\def\draftcontrol{1}}
\global\def\draftcontrol{0}

\makeatother

\newcommand{\alxy}[1]{\begin{aligned}\xymatrix{#1}\end{aligned}}
\newcommand{\alxydim}[2]{\begin{aligned}\xymatrix#1{#2}\end{aligned}}
\renewcommand{\to}{\!\xymatrix@C=0.5cm{\ar[r] &}}
\renewcommand{\mapsto}{\!\xymatrix@C=0.5cm{\ar@{|->}[r] &}\!}
\renewcommand{\Rightarrow}{\!\xymatrix@C=0.5cm{\ar@{=>}[r] &}\!}
\newcommand{\incl}{\!\xymatrix@C=0.5cm{\ar@{^(->}[r] &}\!}
\renewcommand{\hookrightarrow}{\incl}
\renewcommand\Leftrightarrow{\!\xymatrix@C=0.5cm{\ar@{<=>}[r] &}\!}

\usepackage[latin1]{inputenc}

\def\quot#1{``#1''}

\setsecnumdepth{1}
\renewcommand\bun{\mathcal{B}\hspace{-0.2mm}un}
\newcommand{\pg}[1]{\paragraph{#1}}
\title{Multiplicative Bundle Gerbes with Connection}
\author{Konrad Waldorf}
\adress{Department Mathematik\\Schwerpunkt Algebra und Zahlentheorie\\Universität
Hamburg\\Bundesstra\ss e 55\\D--20146 Hamburg
}
\preprint{Hamb. Beitr. Math. Nr. 308\\ZMP-HH/08-07}

\kohyp

\def\mytitle{\zmptitle}

\usepackage{amstext}

\begin{document}

\maketitle

\begin{abstract}
Multiplicative bundle gerbes are  gerbes over a Lie group which are compatible with the group structure. In this article connections on such bundle gerbes are introduced and studied.  It is shown that multiplicative bundle gerbes with connection furnish geometrical constructions of the following objects: smooth central extensions of  loop groups,  Chern-Simons actions for arbitrary gauge groups, and  symmetric bi-branes for  WZW models with topological defect lines. 

\end{abstract}

\tocsection{Introduction}

Over every smooth manifold one finds the following sequence of geometrical objects: smooth $\ueins$-valued functions, principal $\ueins$-bundles, bundle gerbes, bundle 2-gerbes, and so on.  If the manifold is a Lie group $G$, it is interesting to consider subclasses of these objects which are compatible with the group structure. Clearly, a smooth function $f: G \to \ueins$ is compatible with the group structure if it is a group homomorphism. We may write this as
\begin{equation*}
 p_1^{*}f \cdot p_2^{*}f = m^{*}f 
\end{equation*}
with $m: G \times G \to G$ the  multiplication and  $p_i: G \times G \to G$  the   projections. More interestingly,  a  principal $\ueins$-bundle $P$ over $G$ is compatible with the group structure, if it is equipped
with a bundle isomorphism
\begin{equation*}
\phi: p_1^{*}P \otimes p_2^{*}P \to m^{*}P
\end{equation*}
over $G \times G$, generalizing the  equation above. Additionally, the isomorphism $\phi$ has to satisfy  a coherence condition over $G \times G \times G$. Not every bundle $P$ admits such isomorphisms, and if it does, there may be different choices. Pairs $(P,\phi)$ are called multiplicative $\ueins$-bundles and are, indeed, an interesting concept: Grothendieck has shown that multiplicative $\ueins$-bundles are the same as central extensions of $G$ by $\ueins$ \cite{grothendieck1}. 

There is a straightforward generalization to bundle gerbes. Basically, a multiplicative bundle gerbe is a bundle gerbe $\mathcal{G}$  over  $G$ together with an isomorphism
\begin{equation*}
\mathcal{M}: p_1^{*}\mathcal{G} \otimes p_2^{*}\mathcal{G} \to m^{*}\mathcal{G}
\end{equation*}
of bundle gerbes over $G\times G$ and several coherence conditions.  Compared to group homomorphisms and multiplicative $\ueins$-bundles, multiplicative bundle gerbes are particularly interesting. Namely, isomorphism classes of bundle gerbes over an arbitrary smooth manifold $M$ are classified by $H^3(M,\Z)$ \cite{murray2}. For an important class of Lie groups, namely compact, simple and simply-connected ones, this classifying group is canonically isomorphic to the integers. Thus, these Lie groups  carry a  canonical family $\mathcal{G}^k$ of bundle gerbes,  $k\in\Z$. It turns out that all these canonical bundle gerbes are multiplicative \cite{carey4}.

This article is concerned with subtleties that arise when one tries to equip a multiplicative bundle gerbe $\mathcal{G}$ with a connection. 
A priori, every bundle gerbe admits a connection, and it seems natural to demand that the isomorphism $\mathcal{M}$ has to be connection-preserving. However, this turns out to be too restrictive.   

In order to see this, let us consider  the canonical bundle gerbes $\mathcal{G}^k$. Every bundle gerbe $\mathcal{G}^k$ carries a canonical connection  characterized uniquely by fixing its curvature to be the closed 3-form
\begin{equation*}
H_k := \frac{k}{6}\left \langle  \theta \wedge [\theta\wedge\theta] \right \rangle\text{.}
\end{equation*}
Here, $\theta$ denotes the left-invariant Maurer-Cartan form on $G$, and $\left \langle -,-  \right \rangle$ is an invariant inner product on the Lie algebra of $G$ normalized such that $H_1$ represents $1\in \Z = H^3(G,\Z)$ in real cohomology. Now, none of the canonical bundle gerbes equipped with its canonical connection  admits a connection-preserving isomorphism $\mathcal{M}$ like above.

This is in fact easy to see: an isomorphism between two bundle gerbes with connection can only preserve the connections if the two curvatures coincide. In case of the isomorphism $\mathcal{M}$ and the curvature 3-forms $H_k$, this is not true: only the weaker identity
\begin{equation*}
p_1^{*}H_k + p_2^{*}H_k = m^{*}H_k + \mathrm{d}\rho
\end{equation*}
is satisfied, for some 2-form $\rho$ on $G \times G$.

In this article we introduce a new  definition of a connection on a multiplicative bundle gerbe  (Definition \ref{def1}),  such  that the canonical bundle gerbes $\mathcal{G}^k$ provide examples. The idea is to include the 2-form $\rho$ into the  structure: $\mathcal{M}$ is no longer required to be a connection-preserving isomorphism, but only a weaker structure, an invertible  bimodule of curvature $\rho$. 

The details of this definition and some examples are the content of Section \ref{sec2}.
In Section \ref{sec3} we introduce a cohomology theory  which  classifies multiplicative bundle gerbes with connection. It can be seen as a slight modification of the simplicial Deligne cohomology of the classifying space $BG$ of $G$. 
By cohomological methods, we show for compact $G$ (Proposition \ref{prop1}):
\begin{enumerate}
\item[(a)]
Every multiplicative bundle gerbe over $G$ admits a connection.

\item[(b)]
Inequivalent choices of connections  are (up to isomorphism) parameterized by 2-forms $\Omega^2(G)$ modulo closed 2-forms which satisfy a certain simplicial integrality condition.
\end{enumerate}
If one keeps the curvature 3-form $H$ and the curvature 2-form $\rho$ fixed, we show  (Proposition \ref{prop3}) for arbitrary Lie groups $G$:
\begin{enumerate}
\item[(c)]
Isomorphism classes of multiplicative bundle gerbes with connection of  curvature $H$ and  2-form $\rho$ are parameterized by  $H^3(BG,\ueins)$. 
\end{enumerate}

In Section \ref{sec4} we describe three geometrical constructions, all starting from a multiplicative bundle gerbe with connection over a Lie group $G$.

1.)
The first construction yields a smooth, central extension of the loop group $LG$ (Theorem \ref{th3}). For this purpose we introduce a monoidal \quot{transgression} functor which takes bundle gerbes with connection over a smooth manifold $M$ to principal $\ueins$-bundles over $LM$. We show that whenever the bundle gerbe is a multiplicative bundle gerbe with connection, the transgressed $\ueins$-bundle is also multiplicative, so that Grothendieck's correspondence applies.
For the canonical bundle gerbe $\mathcal{G}^k$, we obtain the dual of the $k$-th power of the universal central extension of $LG$  (Corollary \ref{co2}).

2.) The second construction associates to every multiplicative bundle gerbe with connection over any Lie group $G$ and every principal $G$-bundle with connection $A$ over a smooth manifold $M$ a bundle 2-gerbe with connection over $M$ (Theorem \ref{th2}). Its holonomy -- evaluated on a  three-dimensional closed oriented manifold -- is (the exponential of) the  Chern-Simons action functional defined by the connection $A$ (Proposition \ref{prop2}). This \emph{identifies} Chern-Simons theories with gauge group $G$ with multiplicative bundle gerbes with connection over $G$. Our classification results  outlined above imply  a classification of Chern-Simons theories for arbitrary Lie groups (Proposition \ref{prop10}). Reduced to compact Lie groups, we reproduce results due to Dijkgraaf and Witten \cite{dijkgraaf1}.

3.) The third construction yields first examples of symmetric bi-branes. D-branes \cite{carey2} and bi-branes \cite{fuchs4} are additional structures for bundle gerbes with connection that extend  their  holonomy  from closed oriented surfaces to more general classes of surfaces, namely ones with boundary and  with defect lines, respectively. Given a  D-brane for a \emph{multiplicative} bundle gerbe with connection over a Lie group $G$, we construct a  bi-brane in the direct product $G \times G$ (Definition \ref{def4}). For applications in conformal field theory, so-called \emph{symmetric} D-branes and \emph{symmetric} bi-branes are particularly important. We show (Proposition \ref{prop5}) that our construction takes symmetric D-branes to symmetric bi-branes.

\pg{Acknowledgements.}

I thank Urs Schreiber, Christoph Schweigert and  Danny Stevenson for helpful discussions, and acknowledge  support from the  Collaborative Research Centre 676 \quot{Particles, Strings and the Early Universe}.

\section{Geometrical Definition}

\label{sec2}

We start with a brief review of bundle gerbes.
Let $M$ be a smooth manifold. A \emph{bundle gerbe} $\mathcal{G}$ over $M$ is a surjective submersion $\pi:Y \to M$ together with a principal $\ueins$-bundle $L$ over $Y^{[2]}$ and an associative isomorphism 
\begin{equation*}
\mu:\pi_{12}^{*}L \otimes \pi_{23}^{*}L \to \pi_{13}^{*}L
\end{equation*}
of  bundles over $Y^{[3]}$ \cite{murray}. Here we have denoted by $Y^{[k]}$ the $k$-fold fibre product of $Y$ over $M$, and by $\pi_{i_1...i_p}: Y^{[k]} \to Y^{[p]}$ the projections on those components that appear  in the index. 
A \emph{connection} on a bundle gerbe $\mathcal{G}$ is a 2-form $C \in \Omega^2(Y)$ -- called  \emph{curving} -- together with a  connection $\omega$ on  $L$ such that the isomorphism $\mu$ is connection-preserving and 
\begin{equation*}
\pi_2^{*}C - \pi_1^{*}C=\mathrm{curv}(\omega)\text{,}
\end{equation*}
where we identify the curvature of $\omega$ with a real-valued 2-form on the base space $Y^{[2]}$ of  $L$. The curvature of a connection on a bundle gerbe is the unique 3-form $H \in \Omega^3(M)$ such that $\pi^{*}H=\mathrm{d}C$.  For a more detailed introduction to bundle gerbes and their connections  the reader is referred to   recent reviews, e.g.  \cite{schweigert2,murray3} and references therein.

In this article,  all bundle gerbes   come with  connections. A class of trivial examples  is provided by 2-forms $\rho \in \Omega^2(M)$. Their surjective submersion is the identity $\id:M \to M$, so that $M^{[k]}$ is canonically diffeomorphic to $M$. The  bundle $L$ is the trivial $\ueins$-bundle $\trivlin_0$ equipped with the trivial  connection $\omega=0$, and the isomorphism $\mu$ is the identity. Finally, the curving $C$ is the given 2-form $\rho$. This bundle gerbe is called the \emph{trivial bundle gerbe} and denoted $\mathcal{I}_{\rho}$. Its curvature is $H=\mathrm{d}\rho$.

Bundle gerbes with connection form a strictly associative  2-category $\mathfrak{BGrb}(M)$. Most importantly, this means that there are  1-morphisms $\mathcal{A}:\mathcal{G} \to \mathcal{H}$ between two bundle gerbes  $\mathcal{G}$ and $\mathcal{H}$ with connection and   2-morphisms $\alpha:\mathcal{A} \Rightarrow \mathcal{A}'$ between those. Basically, the 1-morphisms are certain principal bundles with connection of a fixed curvature, defined over the fibre product of the surjective submersions of the two bundle gerbes. The 2-morphisms are connection-preserving bundle morphisms between those.
The  precise definitions can be found in \cite{waldorf1}; here we only need to recall some abstract properties. 

Like in every 2-category, a 1-morphism $\mathcal{A}:\mathcal{G} \to \mathcal{H}$ is called \emph{invertible} or 1-\emph{iso}morphism, if there exists another   1-morphism $\mathcal{A}^{-1}: \mathcal{H} \to \mathcal{G}$ in the opposite direction together with 2-isomorphisms $\mathcal{A} \circ \mathcal{A}^{-1} \cong \id_{\mathcal{H}}$ and $\mathcal{A}^{-1} \circ \mathcal{A} \cong \id_{\mathcal{G}}$. The 1-morphisms between $\mathcal{G}$ and $\mathcal{H}$ and all 2-morphisms between those form a category $\mathfrak{Hom}(\mathcal{G},\mathcal{H})$, and the restriction to 1-isomorphisms is a full subcategory $\mathfrak{Iso}(\mathcal{G},\mathcal{H})$. 

\begin{proposition}[\cite{waldorf1}, Sec. 3]
\label{prop4}
For 2-forms $\rho_1,\rho_2\in\Omega^2(M)$ there is a canonical equivalence \begin{equation*}
\bun: \mathfrak{Iso}(\mathcal{I}_{\rho_1},\mathcal{I}_{\rho_2}) \to \ueins\text{-}\mathfrak{Bun}^{\nabla}_{\rho_2 - \rho_1}(M)\text{,} 
\end{equation*}
between the isomorphisms between two trivial bundle gerbes and the category of $\ueins$-bundles over $M$ with connection of fixed curvature $\rho_2-\rho_1$.
\end{proposition}  

This equivalence is useful in the following situation. A 1-isomorphism $\mathcal{T}:\mathcal{G}\to \mathcal{I}_{\rho}$ is called \emph{trivialization} of $\mathcal{G}$. If   $\mathcal{T}_1:\mathcal{G} \to \mathcal{I}_{\rho_1}$ and $\mathcal{T}_2:\mathcal{G} \to \mathcal{I}_{\rho_2}$ are two trivializations of the same bundle gerbe, one obtains a principal $\ueins$-bundle $\bun(\mathcal{T}_2 \circ \mathcal{T}_1^{-1})$ with connection of curvature $\rho_2-\rho_1$. Thus, two trivializations \quot{differ} by a $\ueins$-bundle.

The 2-category $\mathfrak{BGrb}(M)$ of bundle gerbes with connection over $M$ has two important additional structures: pullbacks and tensor products \cite{waldorf1}. The tensor unit is the trivial bundle gerbe $\mathcal{I}_0$. Let us make the following observation:  a flat $\ueins$-bundle $L$  over $M$ corresponds to a 1-isomorphism $L:\mathcal{I}_0 \to \mathcal{I}_0$ under the equivalence of Proposition \ref{prop4}. Its tensor product with some 1-isomorphism  $\mathcal{A}:\mathcal{G} \to \mathcal{H}$ yields a new 1-isomorphism $L \otimes \mathcal{A}:\mathcal{G} \to \mathcal{H}$. This defines a functor
\begin{equation}
\label{36}
\otimes : \ueins\text{-}\mathfrak{Bun}^{\nabla}_0(M) \times \mathfrak{Iso}(\mathcal{G},\mathcal{H}) \to \mathfrak{Iso}(\mathcal{G},\mathcal{H}) \text{.}
\end{equation}
It exhibits the category $\mathfrak{Iso}(\mathcal{G},\mathcal{H})$ as a module  over the monoidal category of flat $\ueins$-bundles. Moreover, this action of flat  bundles  on the isomorphisms between two bundle gerbes $\mathcal{G}$ and $\mathcal{H}$ is \quot{free and transitive} in the sense that the induced functor
\begin{equation}
\label{38}
\ueins\text{-}\mathfrak{Bun}^{\nabla}_0(M) \times \mathfrak{Iso}(\mathcal{G},\mathcal{H}) \to \mathfrak{Iso}(\mathcal{G},\mathcal{H}) \times \mathfrak{Iso}(\mathcal{G},\mathcal{H})
\end{equation}
which sends the pair $(L,\mathcal{A})$ to  $(L \otimes \mathcal{A},\mathcal{A})$, is an equivalence of categories, see \cite{schreiber1}, Lemma 2.

Particular 1-morphisms are bimodules  \cite{fuchs4}. If $\mathcal{G}$ and $\mathcal{H}$ are bundle gerbes with connection over $M$, a $\mathcal{G}$-$\mathcal{H}$-\emph{bimodule} is a 1-morphism
\begin{equation*}
\mathcal{A}: \mathcal{G} \to \mathcal{H} \otimes \mathcal{I}_{\rho}\text{.}
\end{equation*}
The 2-form $\rho$ is called the \emph{curvature} of the bimodule.
A bimodule is called \emph{invertible}, if the 1-morphism $\mathcal{A}$ is invertible. A \emph{bimodule morphism} is just a 2-morphism between the respective 1-morphisms.

\begin{remark}
\label{re2}
The set of bundle gerbes with connection subject  to the equivalence relation $\mathcal{G} \sim \mathcal{H}$ if  there exists an invertible $\mathcal{G}$-$\mathcal{H}$-bimodule, is in bijective correspondence with the set of isomorphism classes of bundle gerbes \emph{without connection}. 
\end{remark}

Let $G$ be a Lie group. There is a  family of smooth maps $G^p \to G^r$  that multiply some of the factors, drop some and leave others untouched. To label these maps we indicate the  prescription by indices. For example:
\begin{equation*}
m_{12,3,46,7}(g_1,g_2,g_3,g_4,g_5,g_6,g_7) := (g_1g_2,g_3,g_4g_6,g_7)\text{.}
\end{equation*}
Particular cases are the multiplication $m_{12}:G^2 \to G$ and the projections $m_{k}:G^p \to G$ to the $k$-th factor. Furthermore, we denote the pullback of some geometric object $X$ along one of the maps $g_I: G^p \to G^r$ by $X_I$.

An $n$-form $\rho\in\Omega^n(G^2)$ will be called \emph{$\Delta$-closed}, if 
\begin{equation}
\label{8}
\rho_{2,3} + \rho_{1,23} = \rho_{1,2} + \rho_{12,3}
\end{equation}
as $n$-forms over $G^3$. In this case we denote the  $n$-form (\ref{8}) by $\rho_{\Delta}$. 
To inure the reader to the notation,  this means
\begin{equation*}
\rho_{\Delta} := m_{2,3}^{*}\rho + m_{1,23}^{*}\rho = m_{1,2}^{*}\rho + m_{12,3}^{*}\rho\text{,}
\end{equation*}
or, at a point $(g_1,g_2,g_3)\in G^3$,
\begin{equation*}
(\rho_{\Delta})_{g_1,g_2,g_3} = \rho_{g_2,g_3} +\rho_{g_1,g_2g_3} = \rho_{g_1,g_2} + \rho_{g_1g_2,g_3}\text{.}
\end{equation*}

\begin{definition}
\label{def1}
A \emph{multiplicative bundle gerbe with connection} over a  Lie group $G$ is a triple $(\mathcal{G},\mathcal{M},\alpha)$ consisting of a bundle gerbe $\mathcal{G}$ with connection over $G$  together with an invertible bimodule
\begin{equation*}
\mathcal{M}: \mathcal{G}_1 \otimes \mathcal{G}_2
\to \mathcal{G}_{12} \otimes \mathcal{I}_{\rho}
\end{equation*}
over $G \times G$, whose curvature $\rho$ is $\Delta$-closed, and a bimodule isomorphism
\begin{equation*}
\alxydim{@R=1.8cm@C=2.1cm}{\mathcal{G}_1 \otimes \mathcal{G}_2 \otimes \mathcal{G}_3 \ar[r]^-{\mathcal{M}_{1,2} \otimes \id} \ar[d]_{\id \otimes \mathcal{M}_{2,3}} & \mathcal{G}_{12} \otimes \mathcal{G}_3 \otimes \mathcal{I}_{\rho_{1,2}} \ar[d]^{\mathcal{M}_{12,3} \otimes \id} \ar@{=>}[dl]|*+{\alpha} \\ \mathcal{G}_1 \otimes \mathcal{G}_{23} \otimes \mathcal{I}_{\rho_{2,3}} \ar[r]_-{\mathcal{M}_{1,23} \otimes \id} & \mathcal{G}_{123} \otimes \mathcal{I}_{\rho_{\Delta}}}
\end{equation*}
over $G \times G \times G$ that satisfies
the   pentagon axiom (Figure \ref{fig1} on page \pageref{fig1}).
\end{definition}

Notice that the diagram is well-defined due to the equality
\begin{equation*}
\mathcal{I}_{\rho_{\Delta}} = \mathcal{I}_{\rho_{2,3}} \otimes \mathcal{I}_{\rho_{1,23}} = \mathcal{I}_{\rho_{1,2}} \otimes \mathcal{I}_{\rho_{12,3}} 
\end{equation*}
which follows since $\rho$ is $\Delta$-closed.  If one forgets the connections and puts $\rho=0$, the above definition reduces consistently to the one of a multiplicative bundle gerbe \cite{carey4}. 

\begin{example}
\label{ex1}
Let $\varphi\in\Omega^2(G)$ be a 2-form on $G$, and  $\mathcal{G}:=\mathcal{I}_{\varphi}$  the associated trivial bundle gerbe over $G$. It can be endowed with a multiplicative structure in the following ways:
\begin{enumerate}
\item[(a)] 
In a trivial way.
We put $\rho := \Delta\varphi:= \varphi_1 -\varphi_{12} + \varphi_2$, this defines a $\Delta$-closed 2-form $\rho$ on $G^2$. We obtain an equality
\begin{equation*}
\mathcal{G}_1 \otimes \mathcal{G}_2 = \mathcal{G}_{12} \otimes \mathcal{I}_{\rho}
\end{equation*}
of bundle gerbes with connection over $G^2$, so that the identity 1-isomorphism $\mathcal{M} := \id_{\mathcal{G}_{12}}$ is an invertible bimodule as required. Together with the identity 2-morphism, this yields a multiplicative bundle gerbe with connection associated to every 2-form on $G$. 

\item[(b)] 
In a non-trivial way involving the following structure: a $\ueins$-bundle 
 $L$ with  connection over $G^2$ and a connection-preserving isomorphism 
\begin{equation*}
\phi: L_{1,2} \otimes L_{12,3} \to L_{2,3} \otimes L_{1,23}
\end{equation*}
of $\ueins$-bundles over $G^3$ satisfying the coherence condition
\begin{equation}
\label{37}
(\phi_{2,3,4} \otimes \id) \circ (\id \otimes \phi_{1,23,4}) \circ (\phi_{1,2,3} \otimes \id) = \phi_{1,2,34}\circ (\id \otimes \phi_{12,3,4})
\end{equation} 
over $G^4$. Notice that the curvature of $L$ is automatically $\Delta$-closed, so that we may put $\rho := \mathrm{curv}(L) - \Delta \varphi$. Using the functor $\bun$ from Proposition \ref{prop4}, we set $\mathcal{M} := \bun^{-1}(L)$ and  $\alpha := \bun^{-1}(\phi)$. These are morphisms as required in Definition \ref{def1}, and \erf{37} implies  the pentagon axiom. 

\end{enumerate}
\end{example}

A  subclass of pairs $(L,\phi)$  as used in Example \ref{ex1} (b) is   provided by  $\Delta$-closed 1-forms  $\psi$  on $G^2$: we set $L := \trivlin_{\psi}$, the trivial $\ueins$-bundle with $\psi$ as connection,  and  $\phi:=\id$. The choice $\psi=0$ reproduces Example \ref{ex1} (a).

\begin{example}
\label{ex4}
Let us now consider the canonical bundle gerbes $\mathcal{G}^k$ with their canonical connections. They are defined over  compact, simple and simply-connected Lie groups $G$ for any  $k\in \Z$. Explicit finite-dimensional, Lie-theoretic constructions are available \cite{gawedzki1,meinrenken1,gawedzki2}; here it will be sufficient to use abstract arguments.  The curvature of $\mathcal{G}^k$ is given by multiples $H_k=k\eta$ of the canonical 3-form
\begin{equation}
\label{21}
\eta := \frac{1}{6} \left \langle  \theta \wedge [ \theta \wedge \theta] \right \rangle \in\Omega^3(G)\text{.}
\end{equation}
Here, $\left \langle  -,- \right \rangle$ is a symmetric bilinear form on the Lie algebra $\mathfrak{g}$ which is normalized such that $\eta$ represents the generator of $H^{3}(G,\Z) = \Z$, and $\theta$ is the left-invariant Maurer-Cartan form on $G$. The canonical 3-form satisfies the identity
\begin{equation}
\label{4}
\eta_1  + \eta_2 = \eta_{12}+ \mathrm{d} \rho
\end{equation}
for 3-forms on $G^2$, where
\begin{equation}
\label{5}
\rho :=  \frac{1}{2}\left \langle  p_1^{*}\theta \wedge p_2^{*}\bar\theta
 \right \rangle 
\end{equation}
for $\bar\theta$  the right-invariant Maurer-Cartan form.
The 2-form $\rho$ is $\Delta$-closed as required. We claim that there exist 1-isomorphisms 
\begin{equation}
\label{34}
 \mathcal{G}_1^k \otimes \mathcal{G}_2^k \to \mathcal{G}^k_{12}\otimes \mathcal{I}_{k\rho}\text{.}
\end{equation}
This comes from the fact that isomorphism classes of bundle gerbes with fixed curvature over a smooth manifold $M$ are parameterized by $H^2(M,\ueins)$, but this cohomology group is  by assumption trivial for $M=G^2$. Indeed, the cur\-va\-tu\-res of the bundle gerbes on both sides of \erf{34} coincide due to (\ref{4}); hence, 1-isomorphisms exist.

Let $\mathcal{M}$ be any choice of such a 1-isomorphism. 
Now we consider the bundle gerbes $\mathcal{H} := \mathcal{G}^k_1 \otimes \mathcal{G}_2^k \otimes \mathcal{G}_3^{k} $ and $\mathcal{K} := \mathcal{G}^k_{123}\otimes \mathcal{I}_{k\rho_{\Delta}}$ with connection over $G^3$. These are the bundle gerbes in the upper left and the lower right corner of the diagram in Definition \ref{def1}. The bimodule isomorphism $\alpha$ that remains to construct is  a morphism between two objects in the category $\mathfrak{Iso}(\mathcal{H},\mathcal{K})$. We recall that this category is a module over  $\ueins\text{-}\mathfrak{Bun}_0^{\nabla}(G^3)$ in a free and transitive way. Since $G$ is simply-connected, all objects in the latter category  are isomorphic, and so are all objects in $\mathfrak{Iso}(\mathcal{K},\mathcal{H})$. Hence $\alpha$ exists. 

Not every choice of $\alpha$ will  satisfy the pentagon axiom, but we can still act with an automorphism of the trivial $\ueins$-bundle $\trivlin_0$ on the choices of $\alpha$ in terms of the functor \erf{36}. These are locally constant functions $G^3 \to \ueins$, and since $G$ is simple -- in particular connected --  just elements of $\ueins$. Now, the pentagon axiom compares compositions of pullbacks of $\alpha$ to $G^4$, namely 
\begin{equation*}
\alpha_l := \alpha_{2,3,4} \circ \alpha_{1,23,4} \circ \alpha_{1,2,3} 
\quad\text{ and }\quad
\alpha_r := \alpha_{1,2,34} \circ \alpha_{12,3,4}\text{.}
\end{equation*}
They differ by the action of a number $z\in \ueins$, say $z \otimes \alpha_l =\alpha_r$. Now consider the new choice $\alpha' := z \otimes \alpha$; this evidently satisfies the pentagon axiom
\begin{equation*}
\alpha_l' = z^3 \otimes \alpha_l  = z^2 \otimes \alpha_r  = \alpha_r'\text{.}
\end{equation*}
So, the canonical bundle gerbes $\mathcal{G}^k$ over a simple, compact and simply-connected Lie group $G$ are examples of multiplicative bundle gerbes with connection. 
\end{example}

\begin{remark}
Any bundle gerbe $\mathcal{G}$ with connection over a smooth manifold $M$ provides holonomies $\mathrm{Hol}_{\mathcal{G}}(\phi) \in \ueins$ for smooth maps $\phi:\Sigma \to M$ defined on a closed oriented surface $\Sigma$. For $(\mathcal{G},\mathcal{M},\alpha)$ a multiplicative bundle gerbe over a Lie group $G$, this holonomy has a particular \quot{multiplicative} property \cite{carey4}: for smooth maps $\phi_1,\phi_2:\Sigma \to G$ it satisfies
\begin{equation}
\label{50}
 \mathrm{Hol}_{\mathcal{G}}(\phi_1) \cdot \mathrm{Hol}_{\mathcal{G}}(\phi_2) =\mathrm{Hol}_{\mathcal{G}}(\phi_1 \cdot \phi_2) \cdot \exp \left (2\pi \im \int_{\Sigma} \Phi^{*}\rho \right )\text{,}
\end{equation}
where $\phi_1 \cdot \phi_2$ is the pointwise product, $\rho$ is the curvature of the bimodule $\mathcal{M}$ and $\Phi: \Sigma \to G \times G$ is defined by $\Phi(s) := (\phi_1(s),\phi_2(s))$.  In the physical literature \erf{50} is known as the \emph{Polyakov-Wiegmann formula}.
\end{remark}

\begin{definition}
\label{def2}
Let $(\mathcal{G},\mathcal{M},\alpha)$ and $(\mathcal{G}',\mathcal{M}',\alpha')$ be two multiplicative bundle gerbes with connection over $G$. A \emph{multiplicative 1-morphism} is a 1-morphism $\mathcal{A}:\mathcal{G} \to \mathcal{G}'$ and a 2-isomorphism
\begin{equation}
\label{39}
\alxydim{@R=1.5cm@C=2cm}{\mathcal{G}_1 \otimes \mathcal{G}_2 \ar[r]^-{\mathcal{M}} \ar[d]_{\mathcal{A}_{1} \otimes \mathcal{A}_2} & \mathcal{G}_{12} \otimes \mathcal{I}_{\rho} \ar@{=>}[dl]|*+{\beta} \ar[d]^{\mathcal{A}_{12}  \otimes \id} \\ \mathcal{G}_1' \otimes \mathcal{G}_2' \ar[r]_-{\mathcal{M}'} & \mathcal{G}_{12}' \otimes \mathcal{I}_{\rho'}}
\end{equation}
such that $\beta$ is compatible with the bimodule morphisms $\alpha$ and $\alpha'$ (Figure \ref{fig2} on page \pageref{fig2}).
\end{definition}

The existence of the 2-isomorphism $\beta$ requires that  the curvatures $\rho$ and $\rho'$ of the bimodules $\mathcal{M}$ and $\mathcal{M}'$ coincide. The composition of two multiplicative 1-morphisms
\begin{equation*}
\alxydim{@C=1.5cm}{(\mathcal{G},\mathcal{M},\alpha) \ar[r]^-{(\mathcal{A},\beta)} & (\mathcal{G}',\mathcal{M}',\alpha') \ar[r]^-{(\mathcal{A}',\beta')} & (\mathcal{G}'',\mathcal{M}'',\alpha'')}
\end{equation*}
is declared to be the 1-morphism $\mathcal{A}' \circ \mathcal{A}: \mathcal{G} \to \mathcal{G}''$ together with the 2-isomorphism which is obtained by putting  diagram \erf{39}  for $\beta$ on top of the one for $\beta'$. 
This composition of multiplicative 1-morphisms  is strictly associative. Thus, by restricting Definition \ref{def2} to invertible 1-morphisms $\mathcal{A}$, one obtains an equivalence relation on the set of multiplicative bundle gerbes with connection over $G$. The set of equivalence classes will be studied in Section \ref{sec3}. 

\begin{example}
\normalfont
\label{ex2}
We return to the multiplicative bundle gerbe $\mathcal{I}_{\varphi}$  from Example \ref{ex1} (b) constructed from a triple $(\varphi,L,\phi)$ of a 2-form $\varphi$, a $\ueins$-bundle $L$ with connection over $G^{2}$ and a certain isomorphism $\phi$. Let $\alpha\in\Omega^1(G)$ be a 1-form from which we produce a new triple $(\varphi',L',\phi)$ consisting of the 2-form  $\varphi' := \varphi + \mathrm{d}\alpha$,  the $\ueins$-bundle  $L' := L \otimes \trivlin_{\Delta\alpha}$, and the same isomorphism $\phi$. Then, there is a multiplicative 1-morphism  between $\mathcal{I}_{\varphi}$ and $\mathcal{I}_{\varphi}'$, whose 1-morphism is given by $\mathcal{A} := \bun^{-1}(\trivlin_{\alpha})$, and whose 2-isomorphism is the identity. 
\end{example}

Concerning the canonical bundle gerbes $\mathcal{G}^k$ from Example \ref{ex4} there is a multiplicative 1-isomorphism between $(\mathcal{G}^k,\mathcal{M},\alpha)$ and $(\mathcal{G}^k,\mathcal{M}',\alpha')$ for all different choices of the bimodule $\mathcal{M}$  and the bimodule morphism $\alpha$. This means that the canonical bundle gerbe $\mathcal{G}^k$ is multiplicative in a unique way (up to multiplicative 1-isomorphisms); see Corollary \ref{co1} below. 

\

\section{Cohomological Classification}

\label{sec3}

We introduce a cohomological description for multiplicative bundle gerbes with connection and derive a number of classification results. Concerning the exponential map of $\ueins$ we  fix the convention that the exponential sequence
\begin{equation*}
\alxydim{@C=1.1cm}{0 \ar[r] & \Z \ar@{^(->}[r] & \R \ar[r]^-{\mathrm{e}^{2\pi\im}} & \ueins \ar[r] & 1}
\end{equation*}
is exact, and we use the differential of $\mathrm{e}^{2\pi\im}$ to identify the Lie algebra of $\ueins$ with $\R$.

\subsection{Deligne Cohomology}

We recall the relation between bundle gerbes with connection and degree two Deligne cohomology. For $n\geq 0$, the \emph{Deligne  complex} $\mathcal{D}^{\bullet}(n)$ \cite{brylinski1} on a smooth manifold $M$  is the sheaf complex
\begin{equation*}
\alxydim{@C=1.1cm}{\ueinssheaf_M \ar[r]^-{\mathrm{dlog}} & \underline{\Omega}^1_M \ar[r]^-{\mathrm{d}} & ... \ar[r]^-{\mathrm{d}} & \underline{\Omega}^n_M\text{.}}
\end{equation*}
Here, $\ueinssheaf_M$ denotes the sheaf of smooth $\ueins$-valued functions, $\underline{\Omega}^k_M$ denotes the sheaf of $k$-forms and $\mathrm{d}$ is the exterior derivative. Finally,
$\mathrm{dlog}$ sends a smooth function $g: U \to \ueins$ to the pullback $g^{*}\theta \in \Omega^1(U)$ of the Maurer-Cartan form $\theta$ on $\ueins$, which is a real-valued 1-form according to the above convention. 

The hypercohomology of the Deligne complex is denoted by $H^k(M,\mathcal{D}(n))$. These cohomology groups can be computed via \v Cech resolutions: for any open cover $\mathscr{U}$ of $M$ one has a complex
\begin{equation}
\label{13}
\mathrm{Del}^{m}(\mathscr{U},n) :=  \bigoplus_{m=p+k} \check C^p(\mathscr{U},\mathcal{D}^k(n))
\end{equation}
whose differential is
\begin{equation}
\label{58}
\mathrm{D}|_{\check C^p(\mathscr{U},\mathcal{D}^k(n))} := \begin{cases}
\delta + (-1)^p \mathrm{d} & \text{for }k>0 \\
\delta + (-1)^p\mathrm{dlog} & \text{else.} \\
\end{cases}
\end{equation}
The Deligne cohomology groups can then be obtained as the direct limit
\begin{equation}
\label{9}
H^m(M,\mathcal{D}(n)) = \lim_{\overrightarrow{\;\;\mathscr{U}\;\;}} \; H^m(\mathrm{Del}^{\bullet}(\mathscr{U},n),\mathrm{D})
\end{equation}
over refinements of open covers.

Of most importance are the groups for $m=n$. $H^0(M,\mathcal{D}(0))$ is the group of smooth $\ueins$-valued functions on $M$. To see what $H^1(M,\mathcal{D}(1))$ is, let $U_i$ be the open sets of $\mathscr{U}$. Then,  transition functions 
\begin{equation*}
g_{ij}: U_i \cap U_j \to \ueins
\end{equation*}
and local connections 1-forms $A_{i}$ of  a principal $\ueins$-bundle with connection define a cocycle $\xi=(g_{ij},A_i)\in\mathrm{Del}^1(\mathscr{U},1)$. A connection-preserving isomorphism  defines a cochain $\eta \in \mathrm{Del}^0(\mathscr{U},1)$ in such a way that $\xi' = \xi + \mathrm{D}(\eta)$. This establishes a bijection between isomorphism classes of $\ueins$-bundles with connection and $H^1(M,\mathcal{D}(1))$, see \cite{brylinski1}, Th. 2.2.11. 

Similarly, $H^2(M,\mathcal{D}(2))$ classifies bundle gerbes with connection. For every bundle gerbe $\mathcal{G}=(\pi,L,C,\mu)$ with connection over $M$ there exists an open cover $\mathscr{U}$ that permits to extract a cocycle $(g,A,B)$ in $\mathrm{Del}^{2}(\mathscr{U},2)$. It consists of smooth functions 
\begin{equation*}
g_{ijk}:U_i \cap U_j \cap U_k \to \ueins
\end{equation*}
coming from the isomorphism $\mu$, of 1-forms $A_{ij}\in\Omega^1(U_i \cap U_j)$ coming from the connection on the  bundle $L$, and of 2-forms $B_i\in \Omega^2(U_i)$ coming from the curving $C$. In terms of its local data, the curvature of the bundle gerbe $\mathcal{G}$ is given by $H|_{U_i} = \mathrm{d}B_i$, and the  \emph{Dixmier-Douady class} $\mathrm{DD}(\mathcal{G}) \in H^3(M,\Z)$ mentioned in the introduction is the image of the \v Cech cohomology class of $g$ under the isomorphism
\begin{equation*}
H^2(M,\ueinssheaf) \cong H^3(M,\Z)\text{.}
\end{equation*}

For any 1-isomorphism $\mathcal{A}:\mathcal{G} \to \mathcal{G}'$ one finds a cochain $\eta\in \mathrm{Del}^{1}(\mathscr{U},2)$ with $\xi' = \xi + \mathrm{D}(\eta)$, and for any 2-isomorphism $\varphi: \mathcal{A} \Rightarrow \mathcal{A}'$ between such 1-isomorphisms a cochain $\alpha\in \mathrm{Del}^{0}(\mathscr{U},2)$ with $\eta' = \eta + \mathrm{D}(\alpha)$. Conversely, one can reconstruct bundle gerbes, 1-isomorphisms and 2-isomorphisms from given cocycles and cochains, respectively. 
This establishes a bijection \cite{murray2}
\begin{equation}
\label{40}
\left\{\begin{array}{c}
\text{Isomorphism classes}\\
\text{of bundle gerbes with}\\
\text{connection over $M$}\\
\end{array}\right\}
\cong H^2(M,\mathcal{D}(2))\text{.}
\end{equation}
A detailed account  of the relation between Deligne cohomology and  geometric objects can be found in literature, e.g. 
\cite{schweigert2} and references therein.

\subsection{A Modification of Simplicial Deligne Cohomology}

The discussion of multiplicative bundle gerbes (without connection) over a Lie group $G$ shows that the cohomology of the classifying space $BG$ is relevant \cite{brylinski3,carey4}. We shall review some aspects that will be important.

One considers the simplicial manifold $G^{\bullet} =\lbrace G^q \rbrace_{q \geq 0}$  with the usual face maps $\Delta_i: G^{q} \to G^{q-1}$ for $0 \leq i \leq q$. In the notation of Section \ref{sec2} these face maps are given by 
\begin{equation*}
\Delta_0 = m_{2,...,q}
\;\;\;\text{, }\;\;\;
\Delta_q = m_{1,...,q-1}
\;\;\;\text{ and }\;\;\;
\Delta_i =m_{1,...,i(i+1),...,q}
\text{ for }
1\leq i < q\text{.}
\end{equation*}
For $A$  an abelian Lie group,  let $\underline{A}_M$ denote the sheaf of
smooth $A$-valued functions on a smooth manifold $M$. The sheaf homomorphisms
\begin{equation}
\label{18}
\Delta := \sum_{i=0}^{q} (-1)^{i} \Delta_i^{*}: \underline{A}_{G^{q-1}} \to \underline{A}_{G^{q}}
\end{equation}
define a complex
\begin{equation}
\alxy{\underline{A}_{\lbrace \ast \rbrace} \ar[r]^-{\Delta} & \underline{A}_G \ar[r]^-{\Delta} & \underline{A}_{G^2} \ar[r]^-{\Delta} & \underline{A}_{G^3} \ar[r]^-{\Delta} & ... }
\end{equation}
of sheaves. Following  \cite{brylinski3} we denote the hypercohomology groups of this complex  by $H^m(BG,\underline{A})$. Indeed, if $A^{\delta}$ is the group $A$ equipped with the discrete topology, $H^m(BG,\underline{A}^{\delta})$ is the singular cohomology $H^m(BG,A)$ of the topological space $BG$.

The same cohomology groups have been considered by Segal \cite{segal4}; they are furthermore  related to the \emph{continuous cohomology} $H^m_{\mathrm{ct}}(G,V)$ of $G$: for $V$ a topological vector space endowed with a continuous $G$-action, this is the  cohomology  of a complex whose cochain groups are the continuous maps from $G^q$ to $V$, and whose differential is a continuous analog of the one of finite group cohomology. In the case that $V$ is finite-dimensional and $G$ acts trivially, 
\begin{equation*}
H^m_{\mathrm{ct}}(G,V) = H^m(BG,\underline{V})\text{,}
\end{equation*}
where $V$ is  on the right hand side considered as an abelian Lie group (\cite{brylinski3}, Prop. 1.3).

The homomorphisms $\Delta$ from \erf{18} generalize to arbitrary sheaves of abelian groups, for example to differential forms,
\begin{equation}
\label{22}
\Delta: \Omega^k(G^q) \to \Omega^k(G^{q+1})\text{.}
\end{equation} 
In Section \ref{sec2} we have called $k$-forms $\rho \in \Omega^k(G^2)$ with $\Delta\rho=0$ \emph{$\Delta$-closed}. We denote the kernel of \erf{22} by $\Omega^k_{\Delta}(G^q)$.

We compute
the  cohomology groups $H^m(BG,\underline{A})$  via \v Cech resolutions. Let $\mathfrak{U} = \lbrace\mathscr{U}_q\rbrace_{q\geq 1}$ be a sequence of open covers of $G^q$, whose index sets form a simplicial set in such a way that
\begin{equation*}
\Delta_k(U^{q}_i) \subset U^{q-1}_{\Delta_k(i)}
\end{equation*}
for $U^q_i$ on open set of $\mathscr{U}_q$. A construction of such sequences with arbitrarily fine open covers $\mathscr{U}_q$ can be found in Section 4 of \cite{tu1}. We form the double complex $\check C^p(\mathscr{U}_q,\underline{A}_{G^q})$
and denote its total complex by $\mathrm{Tot}^{m}_{\Delta}(\mathfrak{U},\underline{A})$, equipped with the differential 
\begin{equation*}
\Delta|_{\check C^p(\mathscr{U}_q,\underline{A}_{G^q})} := (-1)^{q}\delta + \Delta\text{.}
\end{equation*}
The cohomology of this total complex computes -- in the direct limit over refinements of sequences of open covers -- the groups $H^m(BG,\underline{A})$.

In order to classify multiplicative bundle gerbes with connection over $G$, we consider the \emph{simplicial Deligne complex} $\mathrm{Del}^\bullet_\Delta(\mathfrak{U},n)$. Its cochain groups are 
\begin{equation}
\label{14}
\mathrm{Del}^m_\Delta(\mathfrak{U},n) := \bigoplus_{m=j+q} \mathrm{Del}^j(\mathscr{U}_{q},n)\text{,}
\end{equation}
and its differential is
\begin{equation}
\label{52}
\mathrm{D}_{\Delta}|_{\mathrm{Del}^j(\mathscr{U}_{q},n)} := (-1)^{q} \mathrm{D} + \Delta\text{.}
\end{equation}
Taking the direct limit over sequences of open covers $\mathfrak{U}$, one obtains the simplicial Deligne cohomology $H^m(BG,\mathcal{D}(n))$ as introduced in \cite{brylinski4,brylinski3}.  Notice that the curvature 2-form $\rho$ of the bimodule $\mathcal{M}$  in the definition of a multiplicative bundle gerbe with connection has not yet a place in the simplicial Deligne complex.

For this purpose, we modify the cochain groups in degree $n+1$,
\begin{equation}
\label{15}
\mathrm{Del}^{n+1}_\Delta(\mathfrak{U},n)^{bi} := \mathrm{Del}^{n+1}_\Delta(\mathfrak{U},n) \oplus \Omega^n_{\Delta}(G^2)\text{,}
\end{equation}
while $\mathrm{Del}^m_\Delta(\mathfrak{U},n)^{bi} := \mathrm{Del}^m_\Delta(\mathfrak{U},n)$ is as before in all other degrees $m\neq n+1$.  
On the additional summand we define the differential by 
\begin{equation*}
\mathrm{D}^{bi}_{\Delta}|_{\Omega_{\Delta}^n(G^2)}: \Omega_{\Delta}^n(G^2) \to \mathrm{Del}^{n}(\mathscr{U}_2,n): \rho \mapsto \begin{cases} (1,0,...,0,-\rho) & n \geq 1 \\
\mathrm{e}^{-2\pi\mathrm{i}\rho} & n=0 \\
\end{cases}
\end{equation*}
and keep $\mathrm{D}^{bi}_{\Delta} := \mathrm{D}_{\Delta}$ in all other cases.
The new differential still satisfies $D^{bi}_{\Delta} \circ \mathrm{D}^{bi}_{\Delta}=0$. The cohomology of this complex -- in the direct limit over sequences of open covers -- is denoted
\begin{equation*}
H^m(BG,\mathcal{D}^{bi}(n)) := \lim_{\overrightarrow{\;\;\mathfrak{U}\;\;}} \; H^m(\mathrm{Del}^{\bullet}_\Delta(\mathfrak{U},n)^{bi},\mathrm{D}^{bi}_{\Delta})\text{.}
\end{equation*}
The most interesting groups are  those with $m=n+1$, of which we shall now  explicitly describe the first ones.

 For $n=0$, we have
\begin{eqnarray*}
\mathrm{Del}^1_\Delta(\mathfrak{U},0)^{bi} &=&\mathrm{Del}^{0}(\mathscr{U}_1,0) \oplus \Omega^0_{\Delta}(G^2)  \\
\mathrm{Del}^2_\Delta(\mathfrak{U},0)^{bi} &=& \mathrm{Del}^{0}(\mathscr{U}_2,0) \oplus \mathrm{Del}^{1}(\mathscr{U}_1,0) 
\end{eqnarray*}
and the differential $\mathrm{D}_{\Delta}^{bi}$ sends a cochain $(g,\rho)$ to
$(\Delta g \cdot \mathrm{e}^{-2\pi\mathrm{i}\rho},-\mathrm{D}g)$. The second component  $\mathrm{D}g=1$ of the cocycle condition infers that $g$ is a globally defined smooth function $g: G \to \ueins$. The first component $\Delta g = \mathrm{e}^{2\pi\mathrm{i}\rho}$ infers that $g$ is a \emph{projective group homomorphism}, i.e. 
\begin{equation}
\label{10}
 g(x)  g(y)= g(xy)\mathrm{e}^{2\pi\mathrm{i}\rho(x,y)}
\end{equation}
for all $x,y\in G$. The condition $\Delta\rho=0$ imposed in \erf{15}  is here important for the compatibility of (\ref{10}) with the associativity of $G$. Summarizing, $H^1(BG,\mathcal{D}^{bi}(0))$ is the group of smooth projective group homomorphisms $(g,\rho)$.

For $n=1$, the the first cochain groups are
\begin{eqnarray*}
\mathrm{Del}^1_\Delta(\mathfrak{U},1)^{bi} &=& \mathrm{Del}^{0}(\mathscr{U}_1,1)
\\
\mathrm{Del}^2_\Delta(\mathfrak{U},1)^{bi} &=& \mathrm{Del}^{0}(\mathscr{U}_2,1)\oplus \mathrm{Del}^{1}(\mathscr{U}_1,1) \oplus  \Omega^1_{\Delta}(G^2)
\\
\mathrm{Del}^3_\Delta(\mathfrak{U},1)^{bi} &=&  \mathrm{Del}^{0}(\mathscr{U}_3,1) \oplus \mathrm{Del}^{1}(\mathscr{U}_2,1) \oplus\mathrm{Del}^{2}(\mathscr{U}_1,1)\text{.}
\end{eqnarray*}
The coboundary of a cochain $h$ in degree one is $(\Delta h,-\mathrm{D}h,0)$, and the one of a cochain $(a,\xi,\rho)$ in degree two is $(\Delta a, \mathrm{D}a + \Delta\xi - (1,\rho), -\mathrm{D}\xi)$. The cocycle condition for $(a,\xi,\rho)$ has the following components: $\mathrm{D}\xi = 0$ means that $\xi$ is a cocycle for  a $\ueins$-bundle $L$ with connection over $G$.  The next condition $\mathrm{D}a + \Delta\xi =(1,\rho)$ means that $a$ defines a connection-preserving bundle isomorphism
\begin{equation*}
\alpha: p_1^{*}L \otimes p_2^{*}L  \to m^{*}L\otimes \trivlin_{\rho}
\end{equation*}
over $G^{2}$, where $\trivlin_{\rho}$ is the trivial $\ueins$-bundle equipped with the connection 1-form $\rho$. The remaining condition $\Delta a=1$ is a coherence condition for $\alpha$ when pulled back to $G^3$. This way, a cocycle in $\mathrm{Del}^2_\Delta(\mathfrak{U},1)^{bi}$ defines a \textit{multiplicative $\ueins$-bundle with connection}. Interesting examples of such bundles can be found in \cite{murray4}. 

Two cocycles $(a,\xi,\rho)$ and $(a',\xi',\rho')$ are equivalent if they differ by  a coboundary $(\Delta h,-\mathrm{D}h,0)$: this means that the 1-forms coincide, $\rho=\rho'$, the component $\xi'=\xi - \mathrm{D}h$ means that $-h$ defines an isomorphism $\nu: L \to L'$ of $\ueins$-bundles with connection, and $a'=a \cdot \Delta h$ means that $\nu$ respects the multiplicative structures defined by $a$ and $a'$. So, cobordant cocycles  in $\mathrm{Del}^2_\Delta(\mathfrak{U},1)^{bi}$ define isomorphic multiplicative $\ueins$-bundles with connection. Summarizing, the cohomology group $H^2(BG,\mathcal{D}^{bi}(1))$ is in bijection to the set of isomorphism classes of multiplicative $\ueins$-bundles with connection over $G$.

The  most important cohomology group in this article is $H^3(BG,\mathcal{D}^{bi}(2))$. The  relevant cochain groups are here
\begin{eqnarray*}
\mathrm{Del}_{\Delta}^2(\mathfrak{U},2)^{bi} &=& \mathrm{Del}^0(\mathscr{U}_2,2) \oplus \mathrm{Del}^1(\mathscr{U}_1,2)
\\
\mathrm{Del}^3_{\Delta}(\mathfrak{U},2)^{bi} &=& \mathrm{Del}^{0}(\mathscr{U}_3,2) \oplus \mathrm{Del}^1(\mathscr{U}_2,2) \oplus \mathrm{Del}^2(\mathscr{U}_1,2) \oplus \Omega^2_{\Delta}(G^2)
\\
\mathrm{Del}^4_{\Delta}(\mathfrak{U},2)^{bi} &=& \mathrm{Del}^{0}(\mathscr{U}_4,2) \oplus \mathrm{Del}^1(\mathscr{U}_3,2) \oplus \mathrm{Del}^2(\mathscr{U}_2,2) \oplus \mathrm{Del}^3(\mathscr{U}_1,2)\text{.}
\end{eqnarray*} 
The coboundary of a cochain $(h,\zeta)$ in degree two is $(\Delta h,\mathrm{D}h + \Delta\zeta,-\mathrm{D}\zeta)$, and the one of a cochain $(a,\mu,\xi,\rho)$ in degree three is
\begin{equation*}
\mathrm{D}_{\Delta}^{bi}(a,\mu,\xi,\rho) = (\Delta a,-\mathrm{D}a + \Delta \mu,\mathrm{D}\mu + \Delta\xi-(1,0,\rho),-\mathrm{D}\xi)\text{.}
\end{equation*}
The last component $\mathrm{D}\xi=0$ of the cocycle condition means that $\xi$ is  a cocycle for a bundle gerbe $\mathcal{G}$ with connection over $G$. The condition $\mathrm{D}\mu + \Delta\xi-(1,0,\rho)=0$ means that $\mu$ defines a 1-isomorphism
\begin{equation*}
\mathcal{M}: p_1^{*}\mathcal{G} \otimes p_2^{*}\mathcal{G} \to m^{*}\mathcal{G} \otimes \mathcal{I}_{\rho} 
\end{equation*}
of bundle gerbes with connection over $G^2$. The condition $\mathrm{D}a = \Delta \mu$ means that $a$ is  a 2-isomorphism
\begin{equation*}
\alpha: (\mathcal{M}_{12,3} \otimes \id) \circ (\mathcal{M}_{1,2} \otimes \id) \Rightarrow (\mathcal{M}_{1,23} \otimes \id) \circ (\id\otimes\mathcal{M}_{2,3}) \text{,}
\end{equation*}
and the last condition $\Delta\alpha=1$ infers the pentagon axiom for $\alpha$. Summarizing, a cocycle $(a,\mu,\xi,\rho)$ defines a multiplicative bundle gerbe with connection over $G$. 

Let $(a,\mu,\xi,\rho)$ and $(a',\mu',\xi',\rho')$ be cocycles differing by the coboundary of a cochain $(h,\zeta)\in \mathrm{Del}_{\Delta}^2(\mathfrak{U},2)^{bi}$. The component $\xi'=\xi - \mathrm{D}\zeta$ means that $-\zeta$ is a 1-isomorphism $\mathcal{A}:\mathcal{G} \to \mathcal{G}'$ between the bundle gerbes $\mathcal{G}$ and $\mathcal{G}'$ coming from the cocycles $\xi$ and $\xi'$. The two forms coincide, $\rho=\rho'$, and the component $\mu' = \mu + \mathrm{D}h + \Delta(-\zeta)$ means that $h$ is a 2-isomorphism
\begin{equation*}
\beta: \mathcal{A}_{12} \circ \mathcal{M} \Rightarrow \mathcal{M}' \circ (\mathcal{A}_1 \otimes \mathcal{A}_2) \text{.}
\end{equation*}
The last component $a' = a \cdot \Delta h$ guarantees the compatibility of $\beta$ with the bimodule morphisms $\alpha$ and $\alpha'$. Hence, cobordant cocycles define isomorphic multiplicative bundle gerbes with connection. 

Conversely, since we can extract local data from bundle gerbes with connection, 1-isomorphisms and 2-isomorphisms, we conclude

\begin{proposition}
\label{prop8}
There exists a canonical bijection
\begin{equation*}
\left\{\begin{array}{c}
\text{Isomorphism classes of}\\
\text{multiplicative bundle gerbes}\\
\text{with connection over $G$}\\
\end{array}\right\}
\cong
H^3(BG,\mathcal{D}^{bi}(2))\text{.}
\end{equation*}
\end{proposition}

We remark that similar bijections for multiplicative gerbes without connection, and with connection but with vanishing curvature $\rho$,  have been shown before \cite{brylinski4,brylinski3,carey4}.

\subsection{Multiplicative Classes}

\def\mch{\mathrm{MC}}

In this section we derive four results on the groups $H^m(BG,\mathcal{D}^{bi}(n))$.
The first result is related to the projection
\begin{equation}
\label{57}
p^{m}: \mathrm{Del}^{m}_\Delta(\mathfrak{U},n)^{bi} \to \mathrm{Tot}^{m}_{\Delta}(\mathfrak{U},\ueinssheaf)
\end{equation}
onto those components of $\mathrm{Del}^{m}_\Delta(\mathfrak{U},n)^{bi}$ which have values in the sheaf $\ueinssheaf$. This is  a chain map into the complex whose cohomology is $H^m(BG,\ueinssheaf)$. Together with the connecting homomorphism 
\begin{equation}
\label{7}
H^m(BG,\ueinssheaf) \to H^{m+1}(BG,\Z)
\end{equation}
of the exponential sequence, the map induced by $p^{m}$ is a homomorphism
\begin{equation}
\label{41}
\mch: H^m(BG,\mathcal{D}^{bi}(n)) \to H^{m+1}(BG,\Z)\text{.}
\end{equation}
If $(\mathcal{G},\mathcal{M},\alpha)$ is a multiplicative bundle gerbe with connection over $G$ that corresponds to a class $\xi\in H^3(BG,\mathcal{D}^{bi}(2))$ under the bijection of Proposition \ref{prop8}, we call $\mch(\xi)\in H^4(BG,\Z)$ its \emph{multiplicative class}.

Another   invariant of a multiplicative bundle gerbe with connection is the pair $(H,\rho)$ consisting of the curvature 3-form of $\mathcal{G}$  and the curvature 2-form of the bimodule $\mathcal{M}$. In general, there is a projection
\begin{equation}
\label{32}
\Omega: H^{n+1}(BG,\mathcal{D}^{bi}(n)) \to \Omega^{n+1}(G) \oplus \Omega^n(G^2)
\end{equation}
whose result is the pair $(\mathrm{d}\omega^n,\rho)$ consisting of the derivative of the top-form in  the Deligne cocycle $(g,\omega^1,...,\omega^n)$ in $\mathrm{Del}^{n}(\mathscr{U}_1,n)$ and the  $n$-form $\rho$ on $G^2$
we have added in (\ref{15}). A pair $(H,\rho)$ of differential forms in the image of the homomorphism $\Omega$ satisfies 
\begin{equation}
\label{42}
\mathrm{d}H=0
\quad\text{, }\quad
\Delta H- \mathrm{d}\rho=0
\quad\text{ and }\quad
\Delta\rho=0\text{.}
\end{equation}
In general, equations \erf{42} mean that $(H,\rho,0,...,0)$ is a cocycle in the total complex of the  \emph{simplicial de Rham double complex }$\Omega^{\bullet}(G^{\bullet})$, whose differential is
\begin{equation*}
\mathrm{d}_{\Delta}|_{\Omega^k(G^q)} = \Delta + (-1)^{1+q}\mathrm{d}\text{,}
\end{equation*}
with the signs chosen compatible with the conventions \erf{58} and \erf{52}.  The cohomology of this complex  computes the real cohomology of $BG$  via \emph{extended de Rham isomorphisms}
$H^{\bullet}(BG,\R) \cong H^{\bullet}(\Omega G)$ \cite{bott3}. 

Our first result on the groups $H^m(BG,\mathcal{D}^{bi}(n))$ is:

\begin{proposition}
\label{prop9}
The image of the multiplicative class $\mch(\xi)$ of a class $\xi$ in $H^{n+1}(BG,\mathcal{D}^{bi}(n))$ under the homomorphisms
\begin{equation}
\label{33}
\alxydim{@C=0.8cm}{H^{n+2}(BG,\Z) \ar[r] & H^{n+2}(BG,\R) \ar[r]^-{\cong}& H^{n+2}(\Omega G)}
\end{equation}
coincides with the class defined by $\Omega(\xi)$. 
\end{proposition}

\begin{proof}
The extended de Rham isomorphisms can be defined as follows. The inclusions \begin{equation*}
 \Omega^k(G^q) \hookrightarrow  \check C^0(\mathscr{U}_q,\underline{\Omega}^k)
\quad\text{ and }\quad
 \check C^p(\mathscr{U}_q,\R) \hookrightarrow  \check C^p(\mathscr{U}_q,\underline{\Omega}^0)
\end{equation*}
define chain maps from the double complexes which compute, respectively, $H^m(\Omega G)$ and $H^m(BG,\R)$, into  the  triple complex $\check C^p(\mathscr{U}_q,\underline{\Omega}^k)$ which has a differential $\mathrm{D}_{\Omega}$ with the signs arranged like the ones in \erf{52}. By definition, two cocycles correspond under the extended de Rham isomorphism $H^m(BG,\R)\cong H^m(\Omega G)$ to each other, if the sum of their inclusions is a coboundary in this triple complex. 

The triple complex $\check C^p(\mathscr{U}_q,\underline{\Omega}^k)$ is similar to the  complex $\mathrm{Del}^\bullet_\Delta(\mathfrak{U},n)^{bi}$, except (a) it has for $k=0$ the sheaf $\underline{\Omega}^0$ instead of $\ueinssheaf$, (b) it is not truncated above $k=n$, and (c) has not the additional $n$-form. Thus, by taking exponentials, truncating and putting zero for the additional $n$-form, we have a chain map from $\check C^p(\mathscr{U}_q,\underline{\Omega}^k)$ to $\mathrm{Del}^\bullet_\Delta(\mathfrak{U},n)$. Conversely, if we assume open covers $\mathscr{U}_q$ on which one can choose smooth logarithms of $\ueins$-valued functions, one has a section of this chain map, which is itself \emph{not} a chain map:
a straightforward computation shows that the coboundary of the image of a cocycle $\xi \in \mathrm{Del}^m_\Delta(\mathfrak{U},n)$ under this section is precisely the sum of the inclusion  of $\Omega(\xi)$ and of the inclusion of a cocycle 
\begin{equation*}
\kappa :={\textstyle\frac{1}{2\pi\im}}(-\delta\mathrm{log}\alpha_1,\Delta\mathrm{log}\alpha_1 + \delta\mathrm{log}\alpha_2,...,\Delta\mathrm{log}\alpha_n)
\end{equation*}
in $\mathrm{Tot}_{\Delta}^m(\mathfrak{U},\R)$, where  $(\alpha_1,...,\alpha_n) = p^m(\xi)$ under the projection \erf{57}.

It remains to notice that $\kappa$ represents the image of $p^{m}(\xi)$ under the connecting homomorphism \erf{7}, and is thus a \v Cech-representative of the multiplicative class of $\xi$. Hence $\mathrm{MC}(\xi)$ and $\Omega(\xi)$ correspond to each other under the extended de Rham isomorphisms.
\end{proof}

As a consequence of Proposition \ref{prop9} we see that any pair $\Omega(\xi)$ of differential forms is contained in   subspaces 
\begin{equation*}
M_{\Z}^{n+1}(G) \subset M_{\R}^{n+1}(G) \subset \Omega^{n+1}(G) \oplus \Omega^n(G^2)\text{,}
\end{equation*}
where $M_{\R}^{n+1}(G)$ consists of those pairs $(H,\rho)$ which satisfy the cocycle conditions (\ref{42}), and $M_{\Z}^{n+1}$ consists of those whose associated  class in $H^{n+2}(\Omega G)$ lies in the image of the integral cohomology under \erf{33}. 

In preparation of the next result it will be useful to consider the following chain map of complexes of sheaves
\begin{equation}
\label{43}
\iota: \ueins \hookrightarrow \mathcal{D}_{\mathrm{d}}^{\bullet}(n): g \mapsto (g,0,...,0)\text{.}
\end{equation}
Here, the  sheaf  of locally constant $\ueins$-valued functions is considered as a  sheaf complex concentrated in degree zero, and $\mathcal{D}_{\mathrm{d}}^{\bullet}(n)$ denotes the Deligne complex  in which the last sheaf $\underline{\Omega}^n$ is replaced by the sheaf $\underline{\Omega}{}^n_{d}$ of \emph{closed} $n$-forms. We recall that $\iota$ is a quasi-isomorphism \cite{brylinski1}, i.e. it induces isomorphisms
\begin{equation}
\label{19}
H^m(M,\ueins) \cong H^m(M,\mathcal{D}_{\mathrm{d}}(n))\text{.}
\end{equation}
This follows from the Poincaré-Lemma; the same is  true in the simplicial context, as we shall see next. We introduce the complex
\begin{equation*}
\mathrm{Del}_{\Delta}^{m}(\mathfrak{U},n)_{\mathrm{d}} :=\bigoplus_{m=p+q+k} \check C^p(\mathscr{U}_q,\mathcal{D}_{\mathrm{d}}^k(n))\text{.}
\end{equation*}

\begin{lemma}
\label{lem1}
For all $m\geq 0$ and $n>0$, the inclusion
\begin{equation*}
\mathrm{Tot}^{m}_{\Delta}(\mathfrak{U},\ueins) \hookrightarrow  \mathrm{Del}_{\Delta}^{m}(\mathfrak{U},n)_{\mathrm{d}}
\end{equation*}
is a quasi-isomorphism, i.e. it induces isomorphisms
\begin{equation*}
H^m(BG,\ueins) \cong H^{m}(BG,\mathcal{D}_{\mathrm{d}}(n))\text{.}
\end{equation*}
\end{lemma}

\begin{proof}
We show that for $n\geq 1$ the inclusions
\begin{equation*}
\mathrm{Del}_{\Delta}^{m}(\mathfrak{U},n-1)_{\mathrm{d}} \incl \mathrm{Del}_{\Delta}^{m}(\mathfrak{U},n)_{\mathrm{d}}
\quad\text{ and }\quad
\mathrm{Tot}^{m}_{\Delta}(\mathfrak{U},\ueins) \incl \mathrm{Del}_{\Delta}^{m}(\mathfrak{U},1)_{\mathrm{d}}
\end{equation*}
are quasi-isomorphisms. Then the claim follows by using the first quasi-isomorphism $(n-1)$ times and then the second one time. 

To see that the first inclusion is a quasi-isomorphism, suppose 
\begin{equation*}
\xi=(\xi_1,...,\xi_m)\in \mathrm{Del}_{\Delta}^{m}(\mathfrak{U},n)_{\mathrm{d}}
\end{equation*}
is a simplicial Deligne cocycle with a Deligne cocycle $\xi_m  \in \mathrm{Del}^{m-1}(\mathscr{U}_1,n)$ on top. If $m\leq n$, $\xi$ contains no $n$-forms and (if at all) only a closed $(n-1)$-form; it is thus also a cocycle in $\mathrm{Del}_{\Delta}^{m}(\mathfrak{U},n-1)_{\mathrm{d}}$. If $m>n$, we have $\xi_m=(\omega_1,...,\omega_m)$ with an $n$-form $\omega_m \in \check C^{m-n-2}(\mathscr{U}_1,\underline{\Omega}^n)$. Assuming a good cover $\mathscr{U}_1$, we can choose $\eta\in\check C^{m-n-2}(\mathscr{U}_1,\underline{\Omega}^{n-1})$ such that $(-1)^{m-n}\mathrm{d}\eta=\omega_m$. 
We can then regard $\eta$ as a cochain in $\mathrm{Del}_{\Delta}^{m-1}(\mathfrak{U},n)_{\mathrm{d}}$, and obtain a new cocycle
\begin{equation*}
\xi' := \xi +  \mathrm{D}_{\Delta}(\eta)\text{,}
\end{equation*}  
which has by construction no $n$-forms supported on the cover $\mathscr{U}_1$, and since it is still a cocycle, only a \emph{closed} $(n-1)$-form.  Now we can repeat this process inductively over all $\mathscr{U}_q$, since the changes there only affect $n$-forms supported on  $\mathscr{U}_{q+1}$. We have then associated to any simplicial Deligne cocycle in $\mathrm{Del}_{\Delta}^{m}(\mathfrak{U},n)_{\mathrm{d}}$ an equivalent one in $\mathrm{Del}_{\Delta}^{m}(\mathfrak{U},n-1)_{\mathrm{d}}$.
On cohomology classes, this is obviously an inverse to the inclusion.

The proof that the second inclusion is a quasi-isomorphism is similar, just that one has to use $\ueins$-valued functions instead of $0$-forms. For such a function, being closed is the same as being locally constant, which is the claim. 
\end{proof}

We also need the following technical Lemma that relates the extended de Rham isomorphism to the quasi-isomorphism  we have just discussed.

\begin{lemma}
\label{lem2}
The composite
\begin{equation*}
\alxydim{@C=0.9cm}{
M_{\R}^{n+1}(G) \ar[r] &  H^{n+2}(BG,\R) \ar[r] & H^{n+2}(BG,\ueins) \ar[r] & H^{n+2}(BG,\mathcal{D}_{\mathrm{d}}(n))}
\end{equation*}
sends a pair $(H,\rho)\in M_{\R}^{n+1}(G)$ to a simplicial Deligne class represented by a cochain of the form $(0,...0,\xi_n,\xi_{n+1})$. Moreover, with $B \in \check C^0(\mathscr{U}_1,\underline{\Omega}^n)$ such that $\mathrm{d}B=H$, one can take
\begin{equation*}
\xi_n=(1,0,...,0,\Delta B + \rho)
\quad\text{ and }\quad
\xi_{n+1}=(1,0,...,0,-\delta B) \text{.}
\end{equation*}
\end{lemma}

\begin{proof}
The cocycle $(H,\rho,0,...,0)$ associated to  $(H,\rho)\in M_{\R}^{n+1}(G)$ corresponds under the extended de Rham isomorphism to a cocycle $(t_{0},...,t_{n+1})$ in $\mathrm{Tot}_{\Delta}^{n+2}(\mathfrak{U},\R)$, i.e. there exists a cochain $\eta$ in the triple complex $\check C^{p}(\mathscr{U}_q,\underline{\Omega}^k)$ with
\begin{equation}
\label{44}
 (t_{0},...,t_{n+1}) + \mathrm{D}_{\Omega}(\eta) = (H,\rho,0,...,0)\text{.}
\end{equation}
Notice that $\iota(\mathrm{e}^{2\pi\im t_0},...,\mathrm{e}^{2\pi\im t_{n+1}})$ is the image of $(H,\rho)$ under the composite we are about to understand. 
The cochain $\eta$ is a collection $\eta_q^k \in \check C^{n+1-q-k}(\mathscr{U}_q,\underline{\Omega}^k)$ and among the relations implied by \erf{44} are 
\begin{equation*}
-\mathrm{d}\eta_1^{n} = H
\quad\text{ and }\quad
\Delta \eta_1^{n} + \mathrm{d}\eta_2^{n-1}=\rho\text{.}
\end{equation*}
Now denote by $\tilde\eta$ the same cochain but with all $\underline{\Omega}^0$-valued parts exponentiated and $\eta_1^n := 0$. Then $\tilde\eta$ is in $\mathrm{Del}_{\Delta}^{n+1}(\mathfrak{U},n)_{\mathrm{d}}$, and
\begin{equation*}
\iota(\mathrm{e}^{2\pi\im t_0},...,\mathrm{e}^{2\pi\im t_{n+1}}) + \mathrm{D}_{\Delta}(\tilde\eta) \in \mathrm{Del}_{\Delta}^{n+2}(\mathfrak{U},n)_{\mathrm{d}} \end{equation*}
is the claimed cocycle $(0,...,0,\xi_n,\xi_{n+1})$ with $B:=-\eta_1^n$.
\end{proof}

Now we are prepared to prove our second result on $H^m(BG,\mathcal{D}(n)^{bi})$. 

\begin{proposition}
\label{prop3}
The inclusion $\iota$ induces isomorphisms
\begin{equation*}
H^m(BG,\ueins) \cong H^m(BG,\mathcal{D}^{bi}(n))
\end{equation*}
for $m \leq n$. For $m=n+1$ the induced map fits into an exact sequence
\begin{equation*}
\alxydim{@C=0.8cm}{0 \ar[r] & H^{n+1}(BG,\ueins) \ar[r]^-{\iota^{*}} & H^{n+1}(BG,\mathcal{D}^{bi}(n)) \ar[r] & M_{\Z}^{n+1}(G) \ar[r] & 0\text{.}}
\end{equation*}
\end{proposition}

\begin{proof}
For a sequence $\mathfrak{U}$ of open covers $\mathscr{U}_q$ of $G^{q}$ as used before, we define the following complex $\Omega^{\bullet}(\mathfrak{U},n)$. We put $\Omega^m(\mathfrak{U},n)=0$ for $m \leq n$,
\begin{equation*}
\Omega^{n+1}(\mathfrak{U},n) := \check C^0(\mathscr{U}_1,\underline{\Omega}_{\mathrm{d}}^{n+1}) \oplus \Omega^{n}(G^2)
\end{equation*}
and
\begin{equation*}
\Omega^{m}(\mathfrak{U},n) := \bigoplus_{m-n=p+q} \check C^{p}(\mathscr{U}_q,\underline{\Omega}{}_{\mathrm{d}}^{n+1})
\end{equation*}
for $m> n+1$. The differential is on a summand $\check C^{p}(\mathscr{U}_q,\underline{\Omega}{}_{\mathrm{d}}^{n+1})$ given by $(-1)^{q}\delta + \Delta$, and on the summand $\Omega^{n}(G^2)$ by $-\mathrm{d}: \Omega^{n}(G^2) \to \check C^0(\mathscr{U}_2,\underline{\Omega}^{n+1}_\mathrm{d})$. 

The complex $\Omega^{\bullet}(\mathfrak{U},n)$ satisfies two purposes. The first is that the cohomology of $\Omega^{\bullet}(\mathfrak{U},n)$ is trivial in degree $m \leq n$ and  is $M_{\R}^{n+1}(G)$
in degree $n+1$. The second is that we have a chain map
\begin{equation}
\label{3}
\mathrm{d}: \mathrm{Del}_{\Delta}^{\bullet}(\mathfrak{U},n)^{bi} \to \Omega^{\bullet}(\mathfrak{U},n) \end{equation} 
which is the derivative on all $n$-forms that occur in $\mathrm{Del}_{\Delta}^{\bullet}(\mathfrak{U},n)^{bi}$, and the identity on the additional summand $\Omega^n(G^2)$ from \erf{15}. We may assume that all open covers $\mathscr{U}_q$ are such that the Poincaré Lemma is true. Then, the chain map \erf{3} is surjective, and its kernel is precisely the complex $\mathrm{Del}_{\Delta}^{\bullet}(\mathfrak{U},n)_{\mathrm{d}}$. Summarizing,
\begin{equation}
\label{60}
\alxydim{}{0 \ar[r] & \mathrm{Del}_{\Delta}^{\bullet}(\mathfrak{U},n)_{\mathrm{d}} \ar[r] & \mathrm{Del}_{\Delta}^{\bullet}(\mathfrak{U},n)^{bi} \ar[r]^-{\mathrm{d}} & \Omega^{\bullet}(\mathfrak{U},n) \ar[r] & 0}
\end{equation}  
is an exact sequence of complexes.

Due to the vanishing of the cohomology of $\Omega^{\bullet}(\mathfrak{U},n)$ in degree $m\leq n$, the long exact sequence in cohomology induced by \erf{60} splits into isomorphisms
\begin{equation*}
H^{m}(BG,\mathcal{D}_{\mathrm{d}}(n)) \cong H^{m}(BG,\mathcal{D}^{bi}(n))\text{,} \end{equation*}
which show together with Lemma \ref{lem1} the first claim. It remains an exact sequence
\begin{equation*}
\alxydim{@C=0.7cm@R=0.8cm}{0 \ar[r] & H^{n+1}(BG,\ueins) \ar[r]
& H^{n+1}(BG,\mathcal{D}^{bi}(n))   \ar[r]
& M_{\R}^{n+1}(G) \ar `r `[0,0] `^d[0,-3]  `[1,-2] [1,-2] &  \\
& H^{n+2}(BG,\mathcal{D}_{\mathrm{d}}(n)) \ar[r]
&  \hdots
&  
}
\end{equation*}
whose connecting homomorphism is precisely the homomorphism described in Lemma \ref{lem2}. Thus, it factors as
\begin{equation}
\label{31}
\alxydim{@C=0.8cm}{H^{n+2}(BG,\R) \ar[r] & H^{n+2}(BG,\ueins) \ar[r] & H^{n+2}(BG,\mathcal{D}_{\mathrm{d}}(n))\text{.}}
\end{equation} 
The  last arrow is an isomorphism by Lemma \ref{lem1}; hence, any element $(H,\rho)\in M_{\R}^{n+1}(G)$ which lies in the kernel of \erf{31} must already lie in the kernel of $H^{n+1}(BG,\R) \to H^{n+2}(BG,\ueins)$, and thus in the image of $H^{n+2}(BG,\Z)$. This means, by definition, $(H,\rho)\in M_{\Z}^{n+1}(G)$, which proves the second claim.
\end{proof}

\begin{corollary}
\label{co1}
For $G$ compact, simple and simply-connected, there exists a (up to isomorphism) unique multiplicative bundle gerbe with connection for every element in $M_\Z^3(G)$.
\end{corollary}

This follows from Proposition \ref{prop3} and the identity $H^4(BG,\Z)\cong \Z$, which implies 
\begin{equation*}
H^3(BG,\ueins) \cong \mathrm{Tor}H^4(BG,\Z)  = 0\text{.}
\end{equation*}

We recall from Example \ref{ex4} that the canonical bundle gerbe $\mathcal{G}^k$ with connection of curvature $H_k$ over $G$ is multiplicative with a bimodule of curvature $k\rho$. In particular $(H_k,k\rho)\in M_\Z^3(G)$. Corollary \ref{co1} shows now that $\mathcal{G}^k$ is multiplicative in a unique way.

In preparation of the third result on the groups $H^{m}(BG,\mathcal{D}(n)^{bi})$ we continue the discussion of the multiplicative class, i.e. the group homomorphism
\begin{equation*}
\mch: H^m(BG,\mathcal{D}^{bi}(n)) \to H^{m+1}(BG,\Z)
\end{equation*}
induced by the chain map $p^{m}$ that projects onto all  components of $\mathrm{Del}^{m}_\Delta(\mathfrak{U},n)^{bi}$ which have values in the sheaf $\ueinssheaf$.  The kernel $\mathrm{ker}(p^{m})$ together with the restriction of $\mathrm{D}_{\Delta}$ is again a complex, so that
\begin{equation}
\label{20}
\alxydim{}{0 \ar[r] & \mathrm{ker}(p^{m}) \ar[r] & \mathrm{Del}^{m}_\Delta(\mathfrak{U},n)^{bi} \ar[r]^-{p^m} &}
\mathrm{Tot}^{m}_{\Delta}(\mathfrak{U},\ueinssheaf) 
\alxydim{}{ \ar[r] & 0}
\end{equation}
is an exact sequence of complexes. 

In the following we restrict our attention to the case $n=2$, which is relevant for multiplicative bundle gerbes with connection. If we denote the cohomology of $\mathrm{ker}(p^{\bullet})$ by $\mathcal{H}^m$, the interesting part of the long exact sequence induced by \erf{20} is:
\begin{equation}
\label{11}
\alxydim{@C=1cm@R=0.8cm}{&& \hdots \ar[r] &  H^2(BG,\ueinssheaf)    \ar `r  `[0,0]^{\omega} `^d[1,-3]  `[1,-2] [1,-2] &  \\
&\mathcal{H}^3 \ar[r]
& H^3(BG,\mathcal{D}^{bi}(2))   \ar[r]
& H^3(BG,\ueinssheaf)  \ar `r `[0,0] `^d[0,-3]  `[1,-2] [1,-2] &  \\
&\mathcal{H}^4 \ar[r]
&  \hdots
&  
}
\end{equation}
The following lemma contains straightforward calculations concerned with this long exact sequence.

\begin{lemma}
\label{lem4}
There are canonical isomorphisms 
\begin{eqnarray*}
\mathcal{H}^3 &\cong&  \big ( \Omega^2(G) \oplus \Omega^1_{\Delta}(G^2) \big ) \;/\;(-\mathrm{d} \oplus \Delta) \Omega^1(G)\\[0.2cm]
\mathcal{H}^4&\cong& \Omega^1_{\mathrm{d},\Delta}(G^3) \;/\; \Delta\Omega^1(G^2)\text{.}
\end{eqnarray*}
Under these identifications, the  connecting homomorphism $\omega$ 
of \erf{11} is given by the composite
\begin{equation*}
H^2(BG,\ueinssheaf) \to H^3(BG,\Z) \to H^3(BG,\R) \cong H^3(\Omega G) \to \mathcal{H}^3\text{,}
\end{equation*}
where the last arrow is the map
\begin{equation*}
H^3(\Omega G) \ni [(\varphi_2,\varphi_1,\varphi_0)] \mapsto (\varphi_2,\varphi_1) \in \Omega^2(G) \oplus \Omega^1_{\Delta}(G^2)\text{.}
\end{equation*}
which is well-defined in the image of $H^2(BG,\ueinssheaf)$. 
The homomorphism $\mathcal{H}^3 \to H^3(BG,\mathcal{D}^{bi}(2))$ sends a pair $(\varphi,\psi)$ to the family $(\xi_0,\xi_1,\xi_2)$ of Deligne cochains $\xi_k \in \mathrm{Del}^k(\mathscr{U}_{3-k},2)$ given by
\begin{equation*}
\xi_0 := 1
\quad\text{, }\quad
\xi_1 :=(1,\psi)
\quad\text{ and }\quad
\xi_2 := (1,0,\varphi)
\end{equation*}
and to the additional 2-form $\rho := \mathrm{d}\psi + \Delta \varphi\in\Omega^2(G^2)$.
\end{lemma}

\

\begin{remark}
\label{re1}
The homomorphism $\mathcal{H}^3 \to H^3(BG,\mathcal{D}^{bi}(2))$ in the sequence (\ref{11}) has the following geometric counterpart under the bijection of Proposition \ref{prop8}: a pair $(\varphi,\psi)$ is mapped to the multiplicative bundle gerbe $(\mathcal{I}_{\varphi},\trivlin_{\psi},\id)$ constructed  in  Example \ref{ex1} (b). If $(\varphi,\psi)$ is changed to $(\varphi - \mathrm{d}\alpha,\psi + \Delta\alpha)$ by a 1-form $\alpha\in\Omega^1(G)$,  the corresponding multiplicative bundle gerbe is related to the previous one by the  multiplicative 1-isomorphism constructed in Example \ref{ex2}. 
\end{remark}

We denote by $\Omega^n_{\mathrm{d},\Delta,\Z}(G)$ the space of $n$-forms $\varphi$ on $G$ which are closed and  $\Delta$-closed, and whose class $(\varphi,0,...,0)$ in $H^{n+1}(\Omega G) \cong H^{n+1}(BG,\R)$ lies in the image of the integral cohomology $H^{n+1}(BG,\Z)$. Equivalently, the pair $(\varphi,0)$ lies in $M_{\Z}^n(G)$.

\begin{proposition}
\label{prop1}
For $G$ compact, the sequence 
\begin{equation*}
\alxydim{}{ 0  \ar[r] & \Omega^2_{\mathrm{d},\Delta,\Z}(G) \ar[r] & \Omega^2(G)\ar[r] & H^3(BG,\mathcal{D}^{bi}(2)) \ar[r]^-{\mch} & H^4(BG,\Z) \ar[r] & 0}
\end{equation*}
is exact, where the third arrow sends a 2-form $\varphi$ to the class of the  bundle gerbe $\mathcal{I}_{\varphi}$, which is  trivially multiplicative with 2-form $\rho = \Delta \varphi$.
\end{proposition}

\begin{proof}
We recall two important results on the cohomology of compact groups $G$. First, the homomorphism  (\ref{7}) is actually an isomorphism (\cite{brylinski3}, Prop. 1.5):
\begin{equation*}
H^3(BG,\ueinssheaf) \cong H^{4}(BG,\Z)\text{.}
\end{equation*}
The second result is due to Bott \cite{bott2}:  the cohomology of the complex
\begin{equation*}
\alxydim{}{\Omega^q(*) \ar[r]^{\Delta} & \Omega^q(G) \ar[r]^{\Delta} & \Omega^q(G^2) \ar[r]^{\Delta} & ...}
\end{equation*}
is given by $H^p(\Omega^q) = H^{p-q}_{\mathrm{ct}}(G,S^q\mathfrak{g}^{*})$, where $S^q\mathfrak{g}^{*}$ is the $q$-th symmetric power of $\mathfrak{g}^{*}$, considered as a $G$-module under the coadjoint action. Since
$H_{\mathrm{ct}}^m(G,V)=0$
for $m>0$ and $G$ compact \cite{stasheff3}, 
\
we have $H^p(\Omega^q) = 0$ for $p\neq q$. It follows that the identifications of Lemma \ref{lem4} simplify to
\begin{equation*}
\mathcal{H}^3\cong\Omega^2(G)\;/\;\mathrm{d}\Omega^1_{\Delta}(G)
\quad\text{ and }\quad
\mathcal{H}^4=0\text{.}
\end{equation*}
All what remains now is to compute the kernel of the map from $\Omega^2(G)$ to $H^3(BG,\mathcal{D}^{bi}(2))$. 

Suppose first that $\varphi \in \Omega^2_{\mathrm{d},\Delta,\Z}(G)$. By definition, there is a corresponding element in $H^3(BG,\Z)\cong H^2(BG,\ueinssheaf)$ which is mapped to $[(\varphi,0,0)] \in  H^3(\Omega G)$. According to Lemma \ref{lem4} it is  a preimage of $\varphi$ under the connecting homomorphism $\omega$. By the exactness of the sequence \erf{11}, it is hence in the kernel. Conversely, suppose $\varphi$ is in the kernel. Then, it must be in the image of $\omega$, which implies by Lemma \ref{lem4} that it comes from a class in $H^3(BG,\Z)$.
\end{proof}

Proposition \ref{prop1}  makes two important claims for multiplicative bundle gerbes with connection over compact Lie groups.  The first is that for any  class  $\xi \in H^4(BG,\Z)$ there exist multiplicative bundle gerbes with connection, whose multiplicative class is $\xi$. The second is that -- unlike  for multiplicative bundle gerbes without connection (see Proposition 5.2 in \cite{carey4}) -- there may be non-isomorphic choices.

\newcommand{\dw}{\tau}

We will derive one further result on the groups $H^{m}(BG,\mathcal{D}^{bi}(n))$.
We recall that for any Lie group $G$ and any abelian group $A$ there is a  homomorphism
\begin{equation}
\label{17}
\dw: H^m(BG,A) \to H^{m-1}(G,A)
\end{equation}
called \emph{transgression} and defined as follows. For a cocycle $\xi$ in the singular cochain complex of $BG$, the pullback $p^{*}\xi$ to the universal bundle $p:EG \to BG$ is -- since $EG$ is contractible -- a coboundary, say $p^{*}\xi = \mathrm{d}\beta$. For $\iota: G \hookrightarrow EG$  the canonical identification of $G$ with the fibre over the base point of $BG$ we have $\mathrm{d}(\iota^{*}\beta)=0$, so that $\iota^{*}\beta$ defines a class $\dw([\xi]) := [\iota^{*}\beta]$. This class is  independent of the choices of the representative $\xi$ and of  $\beta$. Classes in the image of the transgression homomorphism are called \textit{transgressive}.

It is interesting to rewrite the transgression homomorphism in terms of the simplicial model of $BG$ we have used before. For this purpose we consider the obvious projection
\begin{equation*}
v:\mathrm{Tot}^{m}_{\Delta}(\mathfrak{U},A) \to \check C^{m-1}(\mathscr{U}_1,A) \end{equation*}
from  the  complex that computes $H^m(BG,A)$ to the complex whose cohomology is $H^{m-1}(G,A)$.
Using the simplicial model for the universal bundle $G \hookrightarrow EG \to BG$ one can explicitly check
\begin{lemma}
\label{lem3}
$\tau = v^{*}$.
\end{lemma}

\

Now we consider the homomorphism
\begin{equation*}
\dw^{\infty}: H^3(BG,\mathcal{D}^{bi}(2)) \to H^2(G,\mathcal{D}(2))
\end{equation*}
which is induced by the   projection $v^{\infty}:\mathrm{Del}^3_\Delta(\mathscr{U},2)^{bi} \to \mathrm{Del}^2(\mathscr{U}_1,2)$. In terms of geometric objects, 
it takes a multiplicative bundle gerbe $(\mathcal{G},\mathcal{M},\alpha)$ with connection and forgets $\mathcal{M}$ and $\alpha$. By construction, $v^{\infty}$ lifts the chain map $v$ from Lemma \ref{lem3}, so that we have

\begin{proposition}
The homomorphism $\dw^{\infty}$ lifts  transgression, i.e. the diagram
\begin{equation*}
\alxydim{@C=1.7cm}{H^3(BG,\mathcal{D}^{bi}(2)) \ar[d]_{\mch} \ar[r]^-{\dw^{\infty}} & H^2(G,\mathcal{D}(2)) \ar[d]^{\mathrm{DD}} \\ H^4(BG,\Z)  \ar[r]_{\dw} & H^3(G,\Z)}
\end{equation*}
is commutative. Here, $\mch$ is the multiplicative class and $\mathrm{DD}$ is the Dixmier-Douady class. 
\end{proposition}

\begin{corollary}
\label{cor3}
A bundle gerbe with connection over a compact Lie group is multiplicative  if and only if its Dixmier-Douady class is transgressive. 
\end{corollary}

This follows from the surjectivity of $\mch$ for $G$ compact, see Proposition \ref{prop1}. Corollary \ref{cor3} extends Theorem 5.8 of \cite{carey4} from bundle gerbes to bundle gerbes with connection.

\begin{example}
\label{ex3}
Consider the universal cover $p: \su{2} \to \so{3}$. We have a commutative diagram
\begin{equation*}
\alxydim{}{H^4(B\so{3},\Z) \ar[r]^-{Bp^{*}} \ar[d]_{\dw_{\so 3}} & H^4(B\su{2},\Z) \ar[d]^{\dw_{\su 2}} \\ H^3(\so{3},\Z) \ar[r]_{p^{*}} & H^3(\su{2},\Z)\text{.}}
\end{equation*}
All four cohomology groups can canonically be identified  with $\Z$. With respect to this identification it is well known that $\dw_{\su{2}}$ is the identity, $\dw_{\so{3}}$ and $p^{*}$ are  multiplication by 2, and $Bp^{*}$ is  multiplication by 4.  Now suppose that $\mathcal{G}$ is a bundle gerbe over $\so{3}$. It follows that $p^{*}\mathcal{G}\cong \mathcal{G}^{2k}$ where $\mathcal{G}^k$ is one of the canonical bundle gerbes over $\su{2}$.  Suppose further that $\mathcal{G}$ is multiplicative with some multiplicative class $\xi\in H^4(B\so{3},\Z)$. It follows that the Dixmier-Douady class of $\mathcal{G}$ is  $2\xi$, and the one of its pullback is $4\xi$. Hence, $p^{*}\mathcal{G} \cong \mathcal{G}^{4k}$. Put differently, only those bundle gerbes over $\su{2}$ whose level is divisible by 4  descend to multiplicative bundle gerbes with connection over $\so{3}$. 

\end{example}

\section{Applications}

\label{sec4}

This section contains three constructions of geometrical objects, all starting from a multiplicative bundle gerbe with connection: central extensions of loop groups, bundle 2-gerbes for Chern-Simons theories and symmetric bi-branes. 

\subsection{Central Extensions of Loop Groups}

\label{sec4_1}

First we construct a principal $\ueins$-bundle $\mathscr{T}_{\mathcal{G}}$  over the loop space $LM$, associated to any bundle gerbe $\mathcal{G}$ with connection over $M$.  We show that in case that $M$ is a Lie group $G$ (not necessarily compact, simple or simply-connected) and  $\mathcal{G}$ is a  multiplicative bundle gerbe with connection,   $\mathscr{T}_{\mathcal{G}}$ is a Fréchet Lie group and a central extension of $LG$ by $\ueins$. Our construction reproduces and extends previous work of Pressley and Segal \cite{pressley1}, of Mickelsson \cite{mickelsson1}, and of Brylinski and McLaughlin \cite{brylinski4}.

We equip the free loop space $LM:=C^{\infty}(S^1,M)$ with its usual Fréchet structure (\cite{hamilton1}, Ex. 4.1.2). Let us recall how the chart neighborhoods are constructed. One identifies loops $\tau: S^1 \to M$ with  sections $\tilde\tau: S^1 \to S^1 \times M$ in the trivial bundle with fibre $M$ over $S^1$. The goal is that  sections are embeddings, so that one has for each $\tau$ a tubular neighborhood $E_{\tau}$ of the image of $\tilde\tau$ in $S^1 \times M$. A  chart neighborhood of $\tau$ is now defined by 
\begin{equation*}
V_{\tau} := \left \lbrace \gamma\in LM \;|\; \mathrm{Im}(\tilde\gamma) \subset E_{\tau} \right \rbrace\text{;}
\end{equation*}
it is diffeomorphic to an open subset of the Fréchet space $\Gamma(S^1,\tau^{*}TM)$. We recall the following standard facts.

\begin{lemma}[e.g. \cite{brylinski1}, Prop. 6.1.1]
\label{lem6}
The holonomy 
\begin{equation*}
\mathrm{Hol}_P:LM \to \ueins
\end{equation*}
of a principal $\ueins$-bundle $P$ with connection over $M$ is a smooth map. Its derivative is 
\begin{equation*}
\mathrm{dlog}(\mathrm{Hol}_P) = \int_{S^1} \mathrm{ev}^{*}F \in \Omega^1(LM)\text{,} \end{equation*}
where $F \in \Omega^2(M)$ is the curvature of $P$, $\mathrm{ev}: LM \times S^1 \to M$ is the evaluation map and $\int_{S^1}$ denotes the integration along the fibre.
\end{lemma}

Let  $\mathcal{G}$
be a bundle gerbe with connection over $M$. We construct the $\ueins$-bundle $\mathscr{T}_{\mathcal{G}}$  over $LM$ by specifying separately its fibres, following the ideas of Section 6.2 of \cite{brylinski1}.
For a loop $\gamma:S^1 \to M$ we consider the category $\mathfrak{Iso}(\gamma^{*}\mathcal{G},\mathcal{I}_0)$ of connection-preserving trivializations of the bundle gerbe $\gamma^{*}\mathcal{G}$ over $S^1$. Such trivializations exist: the Deligne  cohomology group that classifies bundle gerbes with connection 
 over $S^1$ in terms of the bijection \erf{40} vanishes: 
\begin{equation*}
H^2(S^1,\mathcal{D}(2)) \stackrel{(\star)}{=} H^2(S^1,\mathcal{D}_{\mathrm{d}}(2)) \stackrel{(*)}{\cong} H^2(S^1,\ueinssheaf)=0\text{,}
\end{equation*}
where $(\star)$ expresses the fact that all 2-forms on $S^1$ are closed and $(*)$ is the quasi-isomorphism \erf{19}. 

We recall from Section \ref{sec2} that the category $\mathfrak{Iso}(\gamma^{*}\mathcal{G},\mathcal{I}_0)$ is a module  over $\ueins\text{-}\mathfrak{Bun}^{\nabla}_0(S^1)$, the category of flat $\ueins$-bundles over $S^1$. On isomorphism classes, this yields an action of the group $\mathrm{Pic}_0^{\nabla}(S^1)$ of isomorphism classes of flat $\ueins$-bundles over $S^1$ on the set  $\mathrm{Iso}(\gamma^{*}\mathcal{G},\mathcal{I}_0)$ of equivalence classes of trivializations. Moreover,  since the equivalence \erf{38} induces a bijection on isomorphism classes, this action is free and transitive. Due to the canonical identifications
\begin{equation}
\label{16}
\mathrm{Pic}^{\nabla}_0(S^1)\cong\mathrm{Hom}(\pi_1(S^1),\ueins)\cong \ueins\text{,}
\end{equation}
we see that the set  $\mathrm{Iso}(\gamma^{*}\mathcal{G},\mathcal{I}_0)$ is a $\ueins$-torsor. This torsor will be the fibre of the $\ueins$-bundle $\mathscr{T}_{\mathcal{G}}$ over the loop $\gamma$, i.e. we set
\begin{equation*}
\mathscr{T}_{\mathcal{G}}:= \bigsqcup_{\gamma\in LM} \mathrm{Iso}(\gamma^{*}\mathcal{G},\mathcal{I}_0)\text{,}
\end{equation*}
and  denote the evident projection by $p:\mathscr{T}_{\mathcal{G}} \to LM$. 

Let us briefly trace back how  $\ueins$ acts on the total space $\mathscr{T}_{\mathcal{G}}$. A number $z\in \ueins$ corresponds under the isomorphism \erf{16} to a flat bundle $P_z$ over $S^1$, characterized up to isomorphism by  $\mathrm{Hol}_{P_z}(S^1) = z$. Using the action \erf{36} of such bundles on 1-isomorphisms, $z$ takes a trivialization $\mathcal{T}$
 to the new trivialization $P_z \otimes \mathcal{T}$ in the same fibre.

Next we define local sections of $p:\mathscr{T}_{\mathcal{G}} \to LM$ over the chart neighborhoods $V_{\tau}$ of $LM$. Our construction differs slightly from the one of \cite{brylinski1}.
Let $E_{\tau}$ be the tubular neighborhood of $\mathrm{Im}(\tilde\tau)$ in $S^1 \times M$ that has been used to define $V_{\tau}$. Let $t: S^1 \times M \to M$ denote the projection on the second factor. Since $E_{\tau}$ is a strong deformation retract of $\mathrm{Im}(\tilde\tau)$, which is in turn diffeomorphic to $S^1$, we see that $H^{2}(E_{\tau},\ueinssheaf)=H^2(S^1,\ueinssheaf)=0$. Hence, every bundle gerbe with connection over $E_{\tau}$ is isomorphic to one of the trivial bundle gerbes $\mathcal{I}_{\rho}$. This allows us to choose a trivialization 
\begin{equation*}
\mathcal{T}_{\tau}:t^{*}\mathcal{G}|_{E_{\tau}} \to \mathcal{I}_{\rho_{\tau}}\text{.}
\end{equation*}
Consider now a loop $\gamma \in V_{\tau}$. By definition $\tilde\gamma$ is a map $\tilde\gamma:S^1 \to E_{\tau}$, and by construction we have $t \circ \tilde\gamma = \gamma$. Thus, we obtain a well-defined section
\begin{equation*}
s_{\mathcal{T}_{\tau}}: V_{\gamma} \to \mathscr{T}_{\mathcal{G}}: \gamma
\mapsto \tilde\gamma^{*}\mathcal{T}_{\tau}\text{.}
\end{equation*}

We use these local sections to equip $\mathscr{T}_{\mathcal{G}}$ with the structure of a Fréchet manifold, again following the lines of \cite{brylinski1}.  Since the fibres of $\mathscr{T}_{\mathcal{G}}$ are $\ueins$-torsors, the sections $s_{\mathcal{T}_{\tau}}$ define bijections
\begin{equation*}
\varphi_{\tau}: V_{\tau} \times \ueins \to p^{-1}(V_{\tau}): (\gamma,z) \mapsto P_z \otimes \tilde\gamma^{*}\mathcal{T}_{\tau}\text{.}
\end{equation*}
These bijections induce a Fréchet manifold structure on each of the open sets $p^{-1}(V_{\tau})\subset \mathscr{T}_{\mathcal{G}}$.
 It remains to show that the transition functions are smooth. For intersecting sets $V_{\tau_1}$ and $V_{\tau_2}$ the trivializations $\mathcal{T}_{\tau_1}$ and $\mathcal{T}_{\tau_2}$ determine by Proposition \ref{prop4} a principal $\ueins$-bundle 
\begin{equation}
\label{51}
P:=\bun(\mathcal{T}_{\tau_1} \circ \mathcal{T}_{\tau_2}^{-1})
\end{equation}
with connection over the intersection $E_{\tau_1} \cap E_{\tau_2}$ with $P \otimes \mathcal{T}_{\tau_2} \cong \mathcal{T}_{\tau_1}$. 
It follows that the transition function $\varphi_{\tau_2}^{-1} \circ \varphi_{\tau_1}$ is given by
\begin{equation*}
 (\gamma,z) \mapsto (\gamma, z\cdot \mathrm{Hol}_P(\tilde \gamma))\text{,}
\end{equation*}
which is smooth by Lemma \ref{lem6}.
The same calculation shows that a choice of different trivializations $\mathcal{T}_{\tau}'$ gives rise to a compatible Fréchet structure. 
Summarizing, we have shown

\begin{proposition}
\label{prop7}
For $\mathcal{G}$ a bundle gerbe with connection over $M$,
$\mathscr{T}_{\mathcal{G}}$ is a principal $\ueins$-bundle over $LM$. 
\end{proposition}

Next we establish important functorial properties of our construction.
We consider a 1-isomorphism $\mathcal{A}:\mathcal{G} \to \mathcal{H}$ between bundle gerbes with connection over $M$, and the associated principal $\ueins$-bundles $\mathscr{T}_{\mathcal{G}}$ and $\mathscr{T}_{\mathcal{H}}$. For a trivialization $\mathcal{T}:\gamma^{*}\mathcal{G} \to \mathcal{I}_0$ of $\mathcal{G}$ over a loop $\gamma\in LM$, we have a trivialization 
\begin{equation*}
\mathcal{T} \circ \gamma^{*}\mathcal{A}^{-1}:\gamma^{*}\mathcal{H} \to \mathcal{I}_0
\end{equation*}
of $\mathcal{H}$ over the same loop.  This is well-defined on equivalence classes of trivializations, and thus defines a map $\mathscr{T}_{\mathcal{A}}:\mathscr{T}_{\mathcal{G}} \to \mathscr{T}_{\mathcal{H}}$. 

\begin{proposition}
\label{prop6}
Let $\mathcal{G}$ and $\mathcal{H}$ be bundle gerbes with connection.
The map
\begin{equation*}
\mathscr{T}_{\mathcal{A}}: \mathscr{T}_{\mathcal{G}} \to \mathscr{T}_{\mathcal{H}}
\end{equation*}
associated to a 1-isomorphism $\mathcal{A}:\mathcal{G} \to \mathcal{H}$ is an isomorphism of principal $\ueins$-bundles over $LM$. Moreover:
\begin{itemize}
\item[(a)]
It respects the composition: if $\mathcal{A}:\mathcal{G} \to \mathcal{H}$ and $\mathcal{B}:\mathcal{H} \to \mathcal{K}$ are 1-isomorphisms, $\mathscr{T}_{\mathcal{B} \circ \mathcal{A}} = \mathscr{T}_{\mathcal{B}} \circ \mathscr{T}_{\mathcal{A}}$. 
\item[(b)]
It respects identities: $\mathscr{T}_{\id_{\mathcal{G}}} = \id_{\mathscr{T}_{\mathcal{G}}}$.

\item[(c)]
The existence of a 2-isomorphism $\beta:\mathcal{A} \Rightarrow \mathcal{B}$ implies that the isomorphisms $\mathscr{T}_{\mathcal{A}}$ and $\mathscr{T}_{\mathcal{B}}$ are equal.

\end{itemize}
\end{proposition}

\begin{proof}
The map $\mathscr{T}_{\mathcal{A}}$ is by definition fibre-preserving. It respects the $\ueins$-actions because the trivializations $P \otimes (\mathcal{T} \circ \gamma^{*}\mathcal{A}^{-1})$ and $(P\otimes \mathcal{T}) \circ \gamma^{*}\mathcal{A}^{-1}$ are equivalent. In order to check that $\mathscr{T}_{\mathcal{A}}$ is smooth, we consider a chart neighborhood $V_{\tau}$ of $LM$,  and the  chart $\varphi_{\tau}$ of $\mathscr{T}_{\mathcal{G}}$ defined by a trivialization $\mathcal{T}_{\tau}: t^{*}\mathcal{G}|_{E_{\tau}} \to \mathcal{I}_{\rho_{\tau}}$ as explained above. We may conveniently choose the trivialization 
\begin{equation*}
\mathcal{T}_{\tau}' := \mathcal{T}_{\tau} \circ t^{*}\mathcal{A}^{-1}: t^{*}\mathcal{H}|_{E_{\tau}} \to \mathcal{I}_{\rho_{\tau}}
\end{equation*}
to determine a chart $\varphi_{\tau}'$ of $\mathscr{T}_{\mathcal{H}}$. Then, \begin{equation*}
\alxydim{}{V_{\tau} \times \ueins \ar[r]^-{\varphi_{\tau}} & p_{\mathcal{G}}^{-1}(V_{\tau}) \ar[r]^-{\mathscr{T}_{\mathcal{A}}} & p_{\mathcal{H}}^{-1}(V_{\tau}) \ar[r]^-{\varphi_{\tau}^{\prime -1}} & V_{\tau} \times \ueins}
\end{equation*}
is the identity, and hence smooth. Assertion
(c) follows directly from the definition. (a) and (b) follow by applying (c) to the canonical 2-isomorphisms $(\mathcal{B} \circ \mathcal{A})^{-1} \cong \mathcal{A}^{-1} \circ \mathcal{B}^{-1}$ and $\mathcal{A} \circ \id_{\mathcal{G}} \cong \mathcal{A}$, respectively.
\end{proof}

The two preceding propositions can be summarized in the following way. We denote by $\mathrm{h}\mathfrak{BGrb}^{\nabla}(M)$ the \quot{homotopy} category whose objects are bundle gerbes with connection over $M$ and whose morphisms are 2-isomorphism classes of 1-isomorphisms. This way we have defined a  functor 
\begin{equation*}
\mathscr{T}:\mathrm{h}\mathfrak{BGrb}^{\nabla}(M) \to \ueins\text{-}\mathfrak{Bun}(LM)
\end{equation*}
that we call \emph{transgression}. 
Moreover, this functor is monoidal. Indeed, bundle isomorphisms
\begin{equation}
\label{49}
 \mathscr{T}_{\mathcal{G}} \otimes \mathscr{T}_{\mathcal{H}} \cong \mathscr{T}_{\mathcal{G} \otimes \mathcal{H}}
\quad\text{ and }\quad
\mathscr{T}_{\mathcal{I}_0}\cong \trivlin
\end{equation}
can be defined as follows. The first  sends a pair $(\mathcal{T}_1,\mathcal{T}_2)$ of trivializations of $\gamma^{*}\mathcal{G}$ and $\gamma^{*}\mathcal{H}$, respectively, to their tensor product $\mathcal{T}_1 \otimes \mathcal{T}_2: \gamma^{*}(\mathcal{G} \otimes \mathcal{H}) \to \mathcal{I}_0$. The second is a particular case of the more general fact that the $\ueins$-bundle $\mathscr{T}_{\mathcal{I}_{\rho}}$ associated to \emph{any} trivial bundle gerbe $\mathcal{I}_{\rho}$ is canonically globally trivializable: a global section of $\mathscr{T}_{\mathcal{I}_{\rho}}$ is defined by $\gamma \mapsto \gamma^{*}\id$, where $\id: \mathcal{I}_{\rho} \to \mathcal{I}_{\rho}$ is the identity isomorphism. It is straightforward to check that these isomorphisms are smooth and that the coherence axioms for monoidal functors are satisfied.
Summarizing, we have shown

\begin{proposition}
Transgression is a monoidal functor
\begin{equation*}
\mathscr{T}:\mathrm{h}\mathfrak{BGrb}^{\nabla}(M) \to \ueins\text{-}\mathfrak{Bun}(LM)\text{.}
\end{equation*}
\end{proposition}

We observe furthermore, for $f:N \to M$ a smooth map, that the identification $\gamma^{*}f^{*}\mathcal{G} = (f \circ \gamma)^{*}\mathcal{G}$ induces natural isomorphisms
\begin{equation}
\label{61}
Lf^{*}\mathscr{T}_{\mathcal{G}} \cong \mathscr{T}_{f^{*}\mathcal{G}}
\end{equation}
of $\ueins$-bundles over $LN$. 

Suppose $\mathcal{G}$ and $\mathcal{G}'$ are the same underlying bundle gerbe, but equipped with different connections. According to Remark \ref{re2}, there exists an invertible bimodule $\mathcal{A}:\mathcal{G} \to \mathcal{G}' \otimes \mathcal{I}_{\rho}$, whose curvature 2-form $\rho$ compensates the difference between the two connections. The transgression of $\mathcal{A}$ defines an isomorphism $\mathscr{T}_{\mathcal{A}}: \mathscr{T}_{\mathcal{G}} \to \mathscr{T}_{\mathcal{G'} \otimes \mathcal{I}_{\rho}} \cong \mathscr{T}_{\mathcal{G'}} \otimes \mathscr{T}_{\mathcal{I}_{\rho}} \cong \mathscr{T}_{\mathcal{G}'}$. Thus, the principal $\ueins$-bundle $\mathscr{T}_{\mathcal{G}}$ depends on the connection on $\mathcal{G}$ only up to canonical isomorphisms.

As a consequence, we have realized a well-defined group homomorphism
\begin{equation}
\label{24}
\mu: H^3(M,\Z) \to H^2(LM,\Z)\text{,}
\end{equation}
which sends the Dixmier-Douady class of a bundle gerbe $\mathcal{G}$ with connection to the first Chern class of the $\ueins$-bundle $\mathscr{T}_{\mathcal{G}}$. 
\begin{lemma}
\label{lem5}
The homomorphism $\mu$ covers  integration along the fibre in de Rham cohomology up to a sign, i.e. the diagram
\begin{equation*}
\alxydim{}{H^3(M,\Z) \ar[r]^-{\mu} \ar[d] & H^2(LM,\Z) \ar[d] \\ H^3_{\mathrm{dR}}(M) \ar[r]_-{-\int_{S^1}} & H_{\mathrm{dR}}^2(LM)}
\end{equation*}
is commutative. 
\end{lemma}

\begin{proof}
We define a connection on $\mathscr{T}_{\mathcal{G}}$ and show that its curvature is minus  the integration over the fibre of the  curvature of $\mathcal{G}$. Let $\tau\in LM$, and let $s_{\mathcal{T}_{\tau}}: V_{\gamma} \to \mathscr{T}_{\mathcal{G}}$ be a local section defined on the neighborhood $V_{\tau}$ from a trivialization $\mathcal{T}_{\tau}: t^{*}\mathcal{G}|_{E_{\tau}} \to \mathcal{I}_{\rho_{\tau}}$ as explained above. Notice that the evaluation map $\mathrm{ev}:LM \times S^1 \to M$ lifts to a commutative diagram
\begin{equation}
\label{48}
\alxydim{}{V_{\tau} \times S^1 \ar[r]^-{\mathrm{ev}_{\tau}} \ar@{^(->}[d] & E_{\tau} \ar[d]^{t} \\ LM \times S^1 \ar[r]_-{\mathrm{ev}} & M\text{.}}
\end{equation}   
We define a local  1-form
\begin{equation*}
A_{\tau} := -  \int_{S^1} \mathrm{ev}_{\tau}^{*}\rho_{\tau} \in \Omega^1(V_{\tau})
\end{equation*}
by integration along the fibre. For each intersection $V_{\tau_1}\cap V_{\tau_2}$, we have the principal $\ueins$-bundle $P$ from \erf{51}, which has curvature $F = \rho_{\tau_1} - \rho_{\tau_2}$. Thus,
\begin{equation*}
A_{\tau_2} - A_{\tau_1} = \int_{S^1} \mathrm{ev}^{*}F =   \mathrm{dlog}(\mathrm{Hol}_P)
\end{equation*}
by Lemma \ref{lem6}. Since the holonomy of $P$ is a transition function for $\mathscr{T}_{\mathcal{G}}$, we conclude that the local 1-forms $A_{\tau}$ define a connection on $\mathscr{T}_{\mathcal{G}}$. 
The curvature of this connection is
\begin{equation*}
\mathrm{d}A_{\tau} = - \int_{S^1} \mathrm{ev}_{\tau}^{*}\mathrm{d}\rho_{\tau} = - \int_{S^1}\mathrm{ev}_{\tau}^{*}t^{*}H  = -  \int_{S^1}\mathrm{ev}^{*}H\text{,}
\end{equation*}
with $H$ the curvature of the bundle gerbe $\mathcal{G}$, where the first equality follows since integration along a closed fibre is a chain map, the second because the two isomorphic bundle gerbes  $t^{*}\mathcal{G}|_{E_{\tau}}$ and $\mathcal{I}_{\rho_{\tau}}$ necessarily have equal curvatures, and the third is due to the commutativity of diagram \erf{48}. 
\end{proof}

In the following we consider the principal $\ueins$-bundle $\mathscr{T}_{\mathcal{G}}$ associated to a \emph{multiplicative} bundle gerbe $(\mathcal{G},\mathcal{M},\alpha)$ with connection over a Lie group $G$. We show that its total space is a central extension of its base space,  the Fréchet Lie group $LG$. 

To do so, we use Grothendieck's correspondence between central extensions and multiplicative $\ueins$-bundles \cite{grothendieck1}. We shall briefly review this correspondence. As mentioned in the introduction, a multiplicative $\ueins$-bundle over a Fréchet Lie group $H$ is a principal $\ueins$-bundle $p:P \to H$ together with a bundle isomorphism
\begin{equation*}
\phi: p_1^{*}P \otimes p_2^{*}P \to m^{*}P
\end{equation*}
over $H \times H$ such that the diagram
\begin{equation}
\label{46}
\alxydim{@C=1.8cm@R=1.2cm}{p_1^{*}P \otimes p_2^{*}P \otimes p_3^{*}P \ar[r]^-{m_{12}^{*}\phi \otimes \id} \ar[d]_{\id \otimes m_{23}^{*}\phi} & m_{12}^{*}P \otimes p_3^{*}P \ar[d]^{m_{12,3}^{*}\phi} \\ p_1^{*}P \otimes m_{23}^{*}P \ar[r]_-{m_{1,23}^{*}\phi} & m_{123}^{*}P}
\end{equation}
over $H \times H \times H$ is commutative. Concerning the various multiplication maps we have used the notation introduced in Section \ref{sec2}.
A central extension of $H$ is obtained by defining the following Fréchet Lie group structure on the total space $P$. 
\begin{enumerate}
\item 
The product is the top row in the commutative diagram
\begin{equation*}
\alxydim{}{P \times P \ar[d]_{p \times p} \ar[r] & p_1^{*}P \otimes p_2^{*}P \ar[d] \ar[r]^-{\phi} & m^{*}P \ar[d] \ar[r] & P \ar[d]^{p} \\ H \times H \ar@{=}[r] & H \times H \ar@{=}[r] \ar@{=}[r] & H \times H \ar[r]_-{m} & H\text{,}}
\end{equation*}
which is  a smooth map and covers the multiplication of $H$. As a consequence of the commutativity of \erf{46}, the product is associative.

\item
To look for the identity element, we restrict our attention to the fibre of $P$ over $1\in H$, where the isomorphism $\phi$ is an isomorphism $\phi_{1,1}:P_1 \otimes P_1 \to P_1$ of $\ueins$-torsors. Any such isomorphism determines an element $e\in P_1$ with $\phi(p,e)=p$ for all $p\in P_1$. Using the commutativity of \erf{46} it is straightforward to see that $e$ is a right and left identity for the product defined by $\phi$. 

\

\item
The inversion of $P$ is defined using the fact that $P$ has a canonical dual bundle $P^{\vee}$, namely  $P^{\vee}:=P$ but with $\ueins$ acting by inverses. It  has also a canonical isomorphism $d: P^{\vee} \otimes P \to \trivlin$ defined by $d(p,p)=1$ for all $p\in P$. Now, the inversion of $P$ is the top row in the diagram
\begin{equation*}
\alxydim{@C=1.1cm}{ P \ar[d] \ar[r]^{\id} &  P^{\vee} \ar[d] \ar[r]^-{e} & P^{\vee} \otimes P_1 \ar[r]^-{\id \otimes j^{*}\phi^{-1}} \ar[d] \ar[r] & P^{\vee} \otimes P \otimes i^{*}P \ar[d] \ar[r]^-{d \otimes \id} & i^{*}P \ar[d] \ar[r] & P \ar[d] \\ H \ar@{=}[r]  & H \ar@{=}[r] & H \ar@{=}[r] & H \ar@{=}[r] & H \ar[r]_{i}  & H\text{,}}
\end{equation*}
in which $i:H \to H$ is the inversion of $H$, $e$ sends $p\in P^{\vee}$ to $(p,e)$, and $j:H \to H \times H$ is the map $t(h):=(h,i(h))$. The inversion defined like this is a smooth map  and covers $i$. One can  check that it provides  right (and thus also left) inverses for the product defined by $\phi$.

\
 
\end{enumerate}
It remains to notice that the map $\iota:\ueins \to P:z \mapsto e.z$ is a diffeomorphism onto its image $P_1$, and that the sequence
\begin{equation*}
1 \to \ueins \to P \to H \to 1
\end{equation*}
is exact; it is hence a central extension of Fréchet Lie groups.

In order to apply this construction to the transgressed principal $\ueins$-bundle $P:=\mathscr{T}_{\mathcal{G}}$ over $H:=LG$, we only need to define the isomorphism $\phi$. This is done using the transgression of the 1-isomorphism $\mathcal{M}$ and the canonical isomorphisms \erf{49} and \erf{61}: we obtain a bundle isomorphism 
\begin{equation*}
\alxydim{@C=1.4cm}{Lp_1^{*}\mathscr{T}_{\mathcal{G}} \otimes Lp_2^{*}\mathscr{T}_{\mathcal{G}} \cong \mathscr{T}_{p_1^{*}\mathcal{G} \otimes p_2^{*}\mathcal{G}} \ar[r]^-{\mathscr{T}_{\mathcal{M}}} & \mathscr{T}_{m^{*}\mathcal{G} \otimes \mathcal{I}_{\rho}} \cong Lm^{*}\mathscr{T}_{\mathcal{G}}}
\end{equation*}
over $LG \times LG$, which  we  denote by $\phi_{\mathcal{M}}$. 

\begin{lemma}
Let $(\mathcal{G},\mathcal{M},\alpha)$ be a multiplicative bundle gerbe with connection over $G$.
Then,  $(\mathscr{T}_{\mathcal{G}},\phi_{\mathcal{M}})$ is a multiplicative principal $\ueins$-bundle over $LG$.
\end{lemma}

\begin{proof}
We have to show that the associated diagram \erf{46} commutes. This is due to the 2-isomorphism $\alpha$ in the structure of the multiplicative bundle gerbe, whose transgression gives by Proposition \ref{prop6} (c) an equality. Explicitly, we obtain a commutative diagram
\begin{equation*}
\alxydim{@R=1cm@C=0.6cm}{Lp_1^{*}\mathscr{T}_{\mathcal{G}} \otimes Lp_2^{*}\mathscr{T}_{\mathcal{G}} \otimes Lp_3^{*}\mathscr{T}_{\mathcal{G}} \ar[dr] \ar[ddd]_{\id\otimes m_{23}^{*}\phi_{\mathcal{M}}} \ar[rrrr]^-{m_{12}^{*}\phi_{\mathcal{M}} \otimes \id} &&&& Lm_{12}^{*}\mathscr{T}_{\mathcal{G}} \otimes Lp_3^{*}\mathscr{T}_{\mathcal{G}} \ar[ddd]^{m_{12,3}^{*}\phi_{\mathcal{M}}} \\ & \hspace{-1cm}\mathscr{T}_{\mathcal{G}_1 \otimes \mathcal{G}_2 \otimes \mathcal{G}_3} \ar[rr]^-{\mathscr{T}_{\mathcal{M}_{1,2} \otimes \id}} \ar[d]_{\mathscr{T}_{\id \otimes \mathcal{M}_{2,3}}} && \mathscr{T}_{\mathcal{G}_{12} \otimes \mathcal{G}_3 \otimes \mathcal{I}_{\rho_{1,2}}} \ar[ur]\hspace{-1cm} \ar[d]^{\mathscr{T}_{\mathscr{T}_{\mathcal{M}_{12,3}}}} & \\ & \hspace{-1cm}\mathscr{T}_{\mathcal{G}_1 \otimes \mathcal{G}_{23} \otimes \mathcal{I}_{\rho_{2,3}}} \ar[rr]_-{\mathscr{T}_{\mathcal{M}_{1,23}}} && \mathscr{T}_{\mathcal{G}_{123} \otimes \mathcal{I}_{\rho_{\Delta}}} \ar[dr]\hspace{-1cm} & \\  Lp_1^{*}\mathscr{T}_{\mathcal{G}} \otimes Lm_{23}^{*}\mathscr{T}_{\mathcal{G}} \ar[ur] \ar[rrrr]_-{m_{1,23}^{*}\phi_{\mathcal{M}}} &&&& Lm_{123}^{*}\mathscr{T}_{\mathcal{G}}}
\end{equation*}
of bundle isomorphisms over $LG \times LG \times LG$: the small subdiagram in the middle is the transgression of $\alpha$, and the other subdiagrams are commutative due to the naturality and the coherence of the isomorphisms \erf{49} and \erf{61}. 
\end{proof}

Summarizing, we have  
\begin{theorem}
\label{th3}
Let $(\mathcal{G},\mathcal{M},\alpha)$ be a  multiplicative bundle gerbe with connection over a Lie group $G$. Then, $\mathscr{T}_{\mathcal{G}}$ is a Fréchet Lie group and 
\begin{equation*}
\alxy{1 \ar[r] & \ueins \ar[r] & \mathscr{T}_{\mathcal{G}} \ar[r]^-{p} & LG \ar[r]
&1}
\end{equation*}
is a  central extension of  $LG$.
\end{theorem}

Theorem \ref{th3} generalizes the geometrical construction of \cite{brylinski4}, Thm. 5.4, from simply-connected Lie groups to arbitrary Lie groups, for which case Brylinski and McLaughlin switch to an abstract cohomological point of view (see \cite{brylinski4}, Thm. 5.1.2). Below we show explicitly that our construction reproduces (for $G$ simply-connected)   the central extensions of Pressley and Segal \cite{pressley1} and of Mickelsson \cite{mickelsson1}.

To start with, let us briefly write out the multiplication on $\mathscr{T}_{\mathcal{G}}$ in terms of trivializations of $\mathcal{G}$. 
Suppose $\mathcal{T}_1$ and $\mathcal{T}_2$ are trivializations representing elements in the fibres of $\mathscr{T}_{\mathcal{G}}$ over loops $\gamma_1$ and $\gamma_2$. Then, their product $\phi_{\mathcal{M}}(\mathcal{T}_1,\mathcal{T}_2)$  is represented by  the trivialization
\begin{equation}
\label{62}
\alxydim{@C=2cm}{(\gamma_1\gamma_2)^{*}\mathcal{G}  \ar[r]^-{\Delta_{\gamma_1,\gamma_2}^{*}\mathcal{M}^{-1}} & \gamma_1^{*}\mathcal{G}
\otimes \gamma_2^{*}\mathcal{G} 
\ar[r]^-{\mathcal{T}_1 \otimes \mathcal{T}_2} & \mathcal{I}_0\text{,}}
\end{equation}
where $\Delta_{\gamma_1,\gamma_2}: S^1 \to G^2$ is the loop $\Delta_{\gamma_1,\gamma_2}(z) := (\gamma_1(z),\gamma_2(z))$.

Next we recall from Example \ref{ex4} that for $G$ compact, simple and simply-connected there exist canonical multiplicative bundle gerbes $\mathcal{G}^k$ with connection for each  $k\in\Z$. Thus, Theorem \ref{th3} produces a family $\mathscr{T}_{\mathcal{G}^k}$ of central extensions of $LG$. The following discussion shows what they are.

In general, central extensions of a (Fréchet) Lie group $H$ by an abelian Lie group $A$ are classified by  $H^2(BH,\underline{A})$ (e.g. \cite{brylinski3}, Prop. 1.6).
According to our description of the cohomology of classifying spaces in terms of \v Cech cohomology we choose open covers $\mathscr{U}_1$ of $H$ and $\mathscr{U}_2$ of $H^2$, which are compatible with the face maps in the sense of Section \ref{sec3}. A cocycle in $H^2(BH,\underline{A})$  consists then of  \v Cech cochains 
\begin{equation*} 
g \in \check C^1(\mathscr{U}_1,\underline{A})
\quad\text{ and }\quad
h\in\check C^0(\mathscr{U}_2,\underline{A})
\end{equation*}
such that the cocycle conditions
\begin{equation*}
\delta g = 1
\quad\text{, }\quad
\Delta g = \delta h
\quad\text{ and }\quad
\Delta h =1
\end{equation*}
are satisfied. The \v Cech cocycle
$g$ is a classifying cocycle for the principal $A$-bundle which underlies the given central extension. Thus, the homomorphism
\begin{equation}
\label{45}
\alxydim{}{H^2(BH,A) \ar[r] &  H^1(H,\underline{A})}
\end{equation}
which is induced by the projection $(g,h) \mapsto g$ takes the class of the central extension  to the  class of the underlying principal bundle. From the cochain
$h$ one can extract the characteristic class of the underlying Lie algebra extension \cite{brylinski3}.

Now we specialize to the central extensions of Theorem \ref{th3}, in which case $H:=LG$ and $A=\ueins$. Here we can combine  \erf{45} with the connecting homomorphism of the exponential sequence and obtain a homomorphism
\begin{equation}
\label{64}
H^2(BLG,\ueinssheaf)\to H^2(LG,\Z)\text{.} 
\end{equation}
If we regard Theorem \ref{th3} as a homomorphism
\begin{equation*}
\tilde\mu: H^3(BG,\mathcal{D}^{bi}(2)) \to H^2(BLG,\ueinssheaf)\text{,}
\end{equation*}
we obtain
immediately
\begin{proposition}
The homomorphism $\tilde\mu$ lifts the homomorphism $\mu$, i.e. 
\begin{equation*}
\alxydim{}{ H^3(BG,\mathcal{D}^{bi}(2)) \ar[d]_{\mathrm{DD}} \ar[r]^-{\tilde\mu} & H^2(BLG,\ueinssheaf) \ar[d]^{\erf{64}} \\  H^3(G,\Z) \ar[r]_-{\mu} & H^2(LG,\Z)}
\end{equation*}
is a commutative diagram.
\end{proposition}

Now we recall (\cite{pressley1}, Prop. 4.4.6) that for $G$ compact, simple and simply-connected, there is a universal central extension characterized such that the first Chern class of the underlying principal $\ueins$-bundle is the image of the generator $1 \in \Z = H^3(G,\Z)$ under the integration over the fibre. Since this generator is the Dixmier-Douady class of $\mathcal{G}^1$, we see from Lemma \ref{lem5} that $\mathscr{T}_{\mathcal{G}^{1}}$ is the dual of the universal central extension.  Thus,
\begin{corollary}
\label{co2}
Let $G$ be compact, simple and simply-connected. Then, the central extension $\mathscr{T}_{\mathcal{G}^k}$ is the dual of the $k$-th power of the universal central extension of the loop group $LG$. 
\end{corollary}

This duality can alternatively be expressed in a  nice geometrical way.
The universal central extension of $LG$ has another realization due to Mickelsson \cite{mickelsson1}, emerging from conformal field theory. It is defined as the set of pairs $(\phi,z)$ consisting of a smooth map $\phi: D^2 \to G$ defined on the unit disc and a complex number $z\in \ueins$, subject to the equivalence relation
\begin{equation*}
(\phi,z) \sim (\phi',z') 
\quad\Leftrightarrow\quad
\phi|_{\partial D^2} = \phi'|_{\partial D^2}
\quad\text{ and }\quad
z' = z \cdot \mathrm{e}^{2\pi\im S_{\mathrm{WZ}}(\phi_{\sharp})}\text{.}
\end{equation*}
Here $\phi_{\sharp}: S^2 \to G$ is the continuous and piecewise smooth map obtained by gluing the domains of $\phi$ and $\phi'$ along their common boundary $S^1$ (the latter with reversed orientation),  and $S_{\mathrm{WZ}}$ is the Wess-Zumino term, which is well-defined  for simple and simply-connected Lie groups. We have a projection 
\begin{equation*}
p:\mathcal{E} \to LG: (\phi,z) \mapsto \phi|_{\partial D^2}\text{,}
\end{equation*}
and one can show that its fibres are  $\ueins$-torsors, and that $\mathcal{E}$ is a locally trivial bundle over $LG$. 

The pairing between $\mathcal{E}$ and $\mathscr{T}_{\mathcal{G}^1}$ that expresses the duality of Corollary \ref{co2} is a bundle isomorphism
\begin{equation}
\label{56}
\mathcal{E} \otimes \mathscr{T}_{\mathcal{G}^1} \to LG \times \ueins
\end{equation}
 over $LG$, which we define as follows. For representatives $(\phi,z)$ and $\mathcal{T}: \gamma^{*}\mathcal{G}^1 \to \mathcal{I}_{0}$ of elements in the fibre over a loop $\gamma$, let $\mathcal{S}: \phi^{*}\mathcal{G} \to \mathcal{I}_{\sigma}$ be any trivialization of the pullback of $\mathcal{G}^1$ to $D^2$, and let $T$ be the $\ueins$-bundle over $\partial D^2$ defined as $T:=\bun(\mathcal{T} \circ \mathcal{S}^{-1}|_{\partial D^2})$. Then, the pairing  \erf{56} is given by
\begin{equation*}
(\phi,z) \otimes \mathcal{T} \mapsto z \cdot \exp \left ( 2\pi\im \int_{D^2} \sigma \right ) \cdot \mathrm{Hol}_{T}(\partial D^2) \in \ueins\text{.}
\end{equation*}
This is actually nothing but the D-brane holonomy  \cite{carey2} for oriented surfaces with boundary, and  hence independent of the choice of $\mathcal{S}$, see \cite{waldorf1}. The choice of another representative $\mathcal{T}'$ leads to an isomorphic  bundle $T'$ with the same holonomies as $T$. For the choice of another representative $(\phi',z')$ let $\mathcal{R}: \phi_{\sharp}^{*}\mathcal{G} \to \mathcal{I}_{\omega}$ be a trivialization with restrictions $\mathcal{S}$ and $\mathcal{S}'$ on the domains of $\phi$ and $\phi'$, respectively. Now, 
\begin{equation*}
\mathrm{e}^{2\pi\im S_{\mathrm{WZ}}(\phi_{\sharp})} \stackrel{(*)}{=} \mathrm{Hol}_{\mathcal{G}^1}(\phi_{\sharp}) \stackrel{(\star)}{=} \exp \left ( 2\pi\im \int_{S^2} \omega \right )= \exp \left ( 2\pi\im \int_{D^2} \sigma - 2\pi\im \int_{D^2} \sigma' \right )
\end{equation*}
where $(*)$ is the  relation between the Wess-Zumino term and the holonomy of the bundle gerbe $\mathcal{G}^1$ which underlies all the applications of bundle gerbes in conformal field theory \cite{carey5} and $(\star)$ is precisely the definition of this holonomy. All together, we see that the pairing \erf{56} is well-defined.

\

It is obvious that \erf{56} is $\ueins$-equivariant; in particular it is an isomorphism. Let us finally equip the $\ueins$-bundle $\mathcal{E}$  with a product, which is defined \cite{mickelsson1} by
\begin{equation*}
(\phi_1, z_1) \cdot (\phi_2, z_2) := (\phi_1\phi_2, z_1z_2 \cdot \exp \left ( 2\pi\im \int_{D^2} \Phi^{*}\rho \right )  )  \text{,}
\end{equation*}
with $\Phi: D^2 \to G \times G$ defined by $\Phi(s) := (\phi_1(s),\phi_2(s))$, and $\rho$  the   2-form \erf{5}. It is left   to the reader to verify that the pairing \erf{56} indeed respects the products on $\mathcal{E}$ and $\mathscr{T}_{\mathcal{G}^1}$.

\subsection{Bundle 2-Gerbes for Chern-Simons Theory}

We construct from a multiplicative bundle gerbe  with connection over a Lie group $G$ and a principal $G$-bundle  with connection $A$ over some smooth manifold $M$ a bundle 2-gerbe $\mathbb{G}$ with connection over $M$. We show that the holonomy of this 2-gerbe around closed oriented three-dimensional manifolds coincides with (the exponential of) the Chern-Simons action for the connection $A$.

Let $G$ be a Lie group and $\mathfrak{g}$ its Lie algebra. Let $p:E \to M$ be a principal $G$-bundle over a smooth manifold $M$, and let $A\in\Omega^1(E,\mathfrak{g})$ be a connection on $E$. We recall that for every invariant  polynomial $P$ on $\mathfrak{g}$ of degree $l$ there exists a canonical invariant $(2l-1)$-form $TP(A)$ on $E$ such that 
\begin{equation}
\label{23}
\mathrm{d}TP(A)=P(\Omega_A^l)\text{,}
\end{equation}
where $\Omega_A := \mathrm{d}A + [A \wedge A]$ is the curvature 2-form of $A$ (\cite{chern1}, Prop. 3.2). Commonly, Chern-Simons theory refers to the study of the form $TP(A)$ in the case $l=2$. In this case,  $TP(A)$ is 
\begin{equation*}
TP(A)= P(  A \wedge \mathrm{d}A)
+ \frac{2}{3}P( A \wedge
[A
\wedge A]  ) \in \Omega^3(E)\text{,}
\end{equation*}
and (\ref{23}) becomes
\begin{equation}
\label{25}
\mathrm{d}TP(A) = p^{*}F_A\text{,}
\end{equation}
where $F_A \in \Omega^4(M)$ is the  Pontryagin 4-form characterized uniquely by the condition that $p^{*}F_A= P( \Omega_A \wedge \Omega_A  )$.

In case that the manifold $M$ is closed, oriented and three-dimensional, and  the principal bundle $E$ admits a global smooth section $s:M \to E$, the \emph{Chern-Simons action} is defined by
\begin{equation}
\label{2}
Z_M(A) := \int_M s^{*}TP(A) \text{.}
\end{equation}
A sufficient condition for the existence of the section $s$ is that $G$ is simply-connected, but one is also interested in the non-simply connected case.  Dijkgraaf and Witten have made the following proposal \cite{dijkgraaf1}. One assumes that there is a four-dimensional compact oriented manifold $B$ with $\partial B=M$, together with a principal $G$-bundle $\tilde E$ with connection $\tilde A$ over $B$ such that $\tilde E|_M=E$ and $\tilde A|_E=\tilde A$.  Then,
\begin{equation}
\label{29}
Z_{M}(A) := \int_B F_{\tilde A}
\end{equation}
replaces the old definition (\ref{2}). The ambiguities coming from different choices of $B$, $\tilde E$ or $\tilde A$ take their values in $\Z$ so that $2\pi\im Z_{M}(A)$ is well-defined in $\ueins$. By Stokes' Theorem and (\ref{25}), the old expression (\ref{2}) is reproduced whenever the section $s$ exists.

The definition of the Chern-Simons action due to Dijkgraaf and Witten is an analog of the definition of the  Wess-Zumino term given by Witten \cite{witten1}. This term could later be identified as the holonomy of a bundle gerbe with connection \cite{carey5}. One advantage of this identification is that the possible Wess-Zumino terms have the same classification as bundle gerbes with connection, which is  a purely geometrical problem \cite{gawedzki2}. 

Motivated by this observation, also the Chern-Simons action (\ref{2}) should be realized as a holonomy; now of a bundle 2-gerbe and taken around the three-manifold $M$. Let us first recall some facts about bundle 2-gerbes.

\begin{definition}[\cite{stevenson2}] 
A \emph{bundle 2-gerbe} over a smooth manifold $M$ is a surjective submersion $\pi:Y \to M$, a bundle gerbe $\mathcal{H}$ over $Y^{[2]}$, a 1-iso\-mor\-phism
\begin{equation*}
\mathcal{E}:\pi_{12}^{*}\mathcal{H} \otimes  \pi_{23}^{*}\mathcal{H} \to \pi_{13}^{*}\mathcal{H}
\end{equation*}
of bundle gerbes over $Y^{[3]}$, and a 2-isomorphism
\begin{equation*}
\alxydim{@C=1.5cm@R=1.2cm}{p_{12}^{*}\mathcal{H} \otimes p_{23}^{*}\mathcal{H} \otimes p_{34}^{*}\mathcal{H} \ar[d]_{p_{123}^{*}\mathcal{E} \otimes \id} \ar[r]^-{\id \otimes p_{234}^{*}\mathcal{E}} & p_{12}^{*}\mathcal{H} \otimes p_{24}^{*}\mathcal{H} \ar@{=>}[dl]|*+{\mu} \ar[d]^{p_{124}^{*}\mathcal{E}} \\ p_{13}^{*}\mathcal{H} \otimes p_{34}^{*}\mathcal{H} \ar[r]_-{p_{134}^{*}\mathcal{E}} & p_{14}^{*}\mathcal{H}}
\end{equation*} 
such that $\mu$ satisfies the natural pentagon axiom. A \emph{connection} on a bundle 2-gerbe is a 3-form $C\in\Omega^3(Y)$ together with a connection on $\mathcal{H}$ of curvature
\begin{equation}
\label{26}
\mathrm{curv}(\mathcal{H}) = \pi_2^{*}C - \pi_1^{*}C\text{,}
\end{equation}
such that $\mathcal{E}$ and $\mu$ are 1- and 2-isomorphism of bundle gerbes with connection. 
\end{definition}

Generalizing the Dixmier-Douady class of a bundle gerbe, every bundle 2-gerbe $\mathbb{G}$ has a characteristic class $\mathrm{CC}(\mathbb{G}) \in H^4(M,\Z)$.
Generalizing the trivial bundle gerbes $\mathcal{I}_{\rho}$ associated to 2-forms $\rho$ on $M$, there are trivial bundle 2-gerbes $\mathbb{I}_{H}$ associated to 3-forms $H\in \Omega^3(M)$ with $\mathrm{CC}(\mathbb{I}_H)=0$.

Suppose that $S$ is a closed oriented three-dimensional manifold and $\phi:S \to M$ is a smooth map. The pullback of any bundle 2-gerbe $\mathbb{G}$ with connection over $M$ along $\phi$ is isomorphic to a trivial bundle 2-gerbe $\mathbb{I}_H$ for some 3-form $H$. Then,
\begin{equation}
\label{28}
\mathrm{Hol}_{\mathbb{G}}(S) := \mathrm{exp}\left (2\pi\mathrm{i} \int_{S} H \right )
\end{equation}
is independent of the choice of $H$, and is called the \emph{holonomy} of $\mathbb{G}$ around $S$. The \emph{curvature} of a bundle 2-gerbe $\mathbb{G}$ with connection is the unique 4-form $\mathrm{curv}(\mathbb{G}) \in \Omega^4(M)$ which satisfies $\pi^{*}\mathrm{curv}(\mathbb{G}) = \mathrm{d}C$. A fundamental relation between the holonomy and the curvature of a bundle 2-gerbe with connection is the following: if $B$ is a compact, four-dimensional, oriented smooth manifold and $\Phi: B \to M$ is a smooth map,
\begin{equation}
\label{30}
\mathrm{Hol}_{\mathbb{G}}(\partial B) = \exp \left (2\pi\im \int_B \Phi^{*}\mathrm{curv}(\mathbb{G}) \right )\text{.}
\end{equation}

\
 
A bundle 2-gerbe  related to Chern-Simons theory has been constructed in \cite{johnson1}. It is then shown (\cite{johnson1}, Prop. 8.2) that this bundle 2-gerbe admits a connection whose holonomy is (the exponential of) the Chern-Simons action \erf{2}. In \cite{carey4} Johnson's bundle 2-gerbe is reproduced as the pullback of a \quot{universal} bundle 2-gerbe over $BG$ obtained using multiplicative bundle gerbes (without connection). 
The goal of this section is to provide a more systematical construction of bundle 2-gerbes \emph{with} connection using  multiplicative bundle gerbes \emph{with} connection.

For preparation, we recall that any principal $G$-bundle $E$ over $M$ defines a simplicial manifold $E^{\bullet}$ whose instances are the fibre products $E^{[k]}$ of $E$ over $M$. There is a canonical simplicial map $g: E^{k} \to G^{k-1}$ into the simplicial manifold $G^{\bullet}$, which extends the \quot{transition function} $g:E^{[2]} \to G$ defined by $x \cdot g(x,y) = y$ for all $(x,y)\in E^{[2]}$. It is useful to recall that the geometric realization 
\begin{equation*}
\xi:=|g|: |E^{\bullet}| \to |G^{\bullet}|
\end{equation*}
of $g$ is a classifying map
for the bundle $E$ under the homotopy equivalence $M \cong |E^{\bullet}|$ and with $BG := |G^{\bullet}|$. 

Differential forms on the simplicial manifold $E^{\bullet}$ arrange into a complex
\begin{equation*}
\alxydim{@C=1.2cm}{0 \ar[r] & \Omega^k(M) \ar[r]^-{p^{*}=\Delta} & \Omega^k(E) \ar[r]^-{\Delta} & \Omega^k(E^{[2]})\ar[r]^-{\Delta} & \Omega^k(E^{[3]}) \ar[r] & ...}
\end{equation*}
whose differential is the alternating sum \erf{18}.  It commutes with the exterior derivative so that \begin{equation*}
\mathrm{d}\Delta TP(A)\stackrel{\text{(\ref{25})}}{=}\Delta(p^{*}F_A)=\Delta^2(F_A)=0\text{.}
\end{equation*}
Hence, $\Delta TP(A)$ defines  a cohomology class in $H^3_{\mathrm{dR}}(E^{[2]})$, and one can calculate that this class coincides with the  class of $g^{*}H$, where 
\begin{equation}
\label{12}
H := \frac{1}{6}P(\theta \wedge[\theta \wedge\theta])\text{.}
\end{equation}
This coincidence can be expressed explicitly by
\begin{equation}
\label{27}
\Delta TP(A)=g^{*}H + \mathrm{d}\omega\text{,}
\end{equation}
where the coboundary term is provided by the 2-form
\begin{equation*}
\omega :=-  P( g^{*} \bar\theta \wedge p_1^{*}A ) \in
\Omega^2(E^{[2]})\text{.}
\end{equation*}
It will be important to notice that there exists a unique 2-form $\rho\in\Omega^2(G^2)$ which satisfies
\begin{equation}
\label{53}
g^{*}\rho + \Delta\omega =0\text{.}
\end{equation}
One can check explicitly that this 2-form is given by 
\begin{equation}
\label{47}
\rho :=  \frac{1}{2}P( p_1^{*}\theta \wedge p_2^{*}\bar\theta)\text{.}
\end{equation}

For the following construction of the  bundle 2-gerbe $\mathbb{CS}_E(\mathcal{G},\mathcal{M},\alpha)$ we assume that the following structure is given:
\begin{enumerate}
\item
A multiplicative bundle gerbe $(\mathcal{G},\mathcal{M},\alpha)$  with connection over some Lie group $G$.

\item
An invariant polynomial $P$ of degree two on the Lie algebra $\mathfrak{g}$, such that the curvature $H$ of $\mathcal{G}$ and the curvature $\rho$ of $\mathcal{M}$ are given by \erf{12} and \erf{47}.

\item
A principal $G$-bundle $E$ with connection $A$ over a smooth manifold $M$.

\end{enumerate}

\begin{example}
For $G$  compact, simple and simply-connected, one can choose one of the canonical bundle gerbes $\mathcal{G}^k$ equipped with their canonical multiplicative structure from Example \erf{ex4}. If  $\left \langle -,-  \right \rangle$ is the invariant bilinear form on $\mathfrak{g}$ normalized like described there, one chooses  $P(X,Y) := k\left \langle X,Y  \right \rangle$. 
\end{example}

The first step in the construction of the bundle 2-gerbe $\mathbb{CS}_E(\mathcal{G},\mathcal{M},\alpha)$ is the bundle gerbe 
\begin{equation*}
\mathcal{H} := g^{*}\mathcal{G} \otimes \mathcal{I}_{\omega}
\end{equation*}
with connection over $E^{[2]}$. Using the commutation relations between the simplicial map $g$ and the projections $E^{[3]} \to E^{[2]}$, namely
\begin{equation}
\label{6}
\Delta_3 \circ g = g \circ p_{12}
\quad\text{, }\quad
\Delta_1 \circ g = g \circ p_{23}
\quad\text{ and }\quad
\Delta_2 \circ g = g \circ p_{13}\text{,}
\end{equation}
one obtains a 1-isomorphism $\mathcal{E}$ of bundle gerbes over $E^{[3]}$ by
\begin{equation*}
\alxydim{@C=1.3cm@R=1.1cm}{p_{12}^{*}\mathcal{H} \otimes p_{23}^{*}\mathcal{H}= g^{*}(p_{1}^{*}\mathcal{G} \otimes p_2^{*}\mathcal{G})
\otimes \mathcal{I}_{p_{13}^{*}\omega+\rho} \ar[r]^-{g^{*}\mathcal{M} \otimes
\id}  & g^{*}m^{*}\mathcal{G} \otimes \mathcal{I}_{p_{13}^{*}\omega}
= p_{13}^{*}\mathcal{H}\text{.}}
\end{equation*}
Finally we define a  2-isomorphism $\mu$ of bundle gerbes over $E^{[4]}$ by 
\begin{equation*}
\alxydim{@C=-1.3cm}{p_{12}^{*}\mathcal{H} \otimes p_{23}^{*}\mathcal{H} \otimes p_{34}^{*}\mathcal{H} \ar@{=}[dr] \ar[ddd]_{p_{123}^{*}\mathcal{E} \otimes \id} \ar[rrrr]^{\id \otimes p_{234}^{*}\mathcal{E}} &&&& p_{12}^{*}\mathcal{H} \otimes p_{24}^{*}\mathcal{H} \ar@{=}[dl] \ar[ddd]^{p_{124}^{*}\mathcal{E}} \\&g^{*}(\mathcal{G}_1 \otimes \mathcal{G}_2 \otimes \mathcal{G}_3 \otimes \mathcal{I}_{\tilde\rho}) \ar[d]_{g^{*}\mathcal{M}_{1,2} \otimes \id} \ar[rr]^{\id \otimes g^{*}\mathcal{M}_{2,3}} &\hspace{4cm}& g^{*}(\mathcal{G}_1 \otimes \mathcal{G}_{23} \otimes \mathcal{I}_{\rho_{1,23}}) \ar[d]^{g^{*}\mathcal{M}_{1,23}} \ar@{=>}[dll]|*+{g^{*}\alpha} &\\&g^{*}(\mathcal{G}_{12}  \otimes \mathcal{G}_3 \otimes \mathcal{I}_{\rho_{12,3}}) \ar[rr]_-{g^{*}\mathcal{M}_{12,3}} &&g^{*}\mathcal{G}_{123} \otimes \mathcal{I}_{p_{14}^{*}\omega} \ar@{=}[dr]  &\\ p_{13}^{*}\mathcal{H} \otimes p_{34}^{*}\mathcal{H} \ar@{=}[ur] \ar[rrrr]_{p_{134}^{*}\mathcal{E}} &&&& p_{14}^{*}\mathcal{H}\text{.}}
\end{equation*}

\begin{theorem}
\label{th2}
The surjective submersion $p:E \to M$, the 3-form $TP(A)$ over $E$, the bundle gerbe $\mathcal{H}$ with connection over $E^{[2]}$, the 1-isomorphism $\mathcal{E}$ and the 2-isomorphism $\mu$ define a bundle 2-gerbe $\mathbb{CS}_E(\mathcal{G},\mathcal{M},\alpha)$ with connection over $M$.  
It has the following properties:
\begin{enumerate}
\item[(a)]
Its characteristic class is the pullback of the multiplicative class of $(\mathcal{G},\mathcal{M},\alpha)$ along a classifying map $\xi: M \to BG$ for the bundle $E$,
\begin{equation*}
\mathrm{CC}(\mathbb{CS}_E(\mathcal{G},\mathcal{M},\alpha)) = \xi^{*}\mch(\mathcal{G},\mathcal{M},\alpha)\text{.}
\end{equation*}

\item[(b)]
Its curvature   is the Pontryagin 4-form of the connection $A$,
\begin{equation*}
 \mathrm{curv}(\mathbb{CS}_E(\mathcal{G},\mathcal{M},\alpha))= F_A\text{.}
\end{equation*}
\end{enumerate}
\end{theorem}

\begin{proof}
To prove that we have defined a bundle 2-gerbe it remains to check the condition (\ref{26}) and the pentagon axiom for the 2-isomorphism $\mu$. The latter follows directly from the  pentagon axiom for $\alpha$, see Figure \ref{fig1}. Condition (\ref{26}) is satisfied:
\begin{equation*}
\mathrm{curv}(\mathcal{H}) = g^{*}H + \mathrm{d}\omega \stackrel{\text{(\ref{27})}}{=} \Delta TP(A)=p_2^{*}TP(A) - p_1^{*}TP(A)\text{.}
\end{equation*}
Property  (a) follows from the fact that -- apart from the forms -- all the structure of the bundle 2-gerbe  is pullback of structure of the multiplicative bundle gerbe along the  simplicial map $g$ which realizes the classifying map $\xi$. (b) follows directly from  (\ref{25}).  
\end{proof}

Let us now study the holonomy of the bundle 2-gerbe from Theorem \ref{th2}. 
\begin{proposition}
\label{prop2}
Let $\mathbb{CS} := \mathbb{CS}_E(\mathcal{G},\mathcal{M},\alpha)$ be the bundle 2-gerbe with connection from Theorem \ref{th2}, associated to a principal $G$-bundle $E$ with connection $A$. 
\begin{enumerate}
\item[(1)]
Let $\phi: S \to M$ be a smooth map where $S$ is a three-dimensional, closed and oriented  manifold, and assume that $E$ has a section $s$ along $\phi$, i.e. a smooth map $s: S \to E$ such that $p\circ s = \phi$. Then,
\begin{equation*}
\mathrm{Hol}_{\mathbb{CS}}(S) = \mathrm{exp}\left (2\pi\mathrm{i} \int_{S} s^{*}TP(A) \right )\text{.}
\end{equation*}
\item[(2)]
Let $\Phi:B \to M$ be a smooth map where $B$ is compact, oriented and four-dimensional. Then,
\begin{equation*}
\mathrm{Hol}_{\mathbb{CS}}(\partial B) = \exp \left ( 2\pi\im \int_B \Phi^{*}F_A \right )\text{.}
\end{equation*}
\end{enumerate}
\end{proposition}

\begin{proof}
In the first case there exists a  trivialization $\mathbb{T}: \phi^{*}\mathbb{CS} \to \mathbb{I}_{s^{*}TP(A)}$ since the surjective submersion of the bundle 2-gerbe $\phi^{*}\mathbb{CS}$ has a section. Then (\ref{28}) proves the assertion. The second case follows from \erf{30}. 
\
\end{proof}

Putting $S=M$  and $\phi=\id$, (1) reproduces the original definition (\ref{2}) of the Chern-Simons action. Putting $B=M$ and  $\Phi=\id$, (2) reproduces the extended definition \erf{29} of Dijkgraaf and Witten. 
This motivates the following definition of a Chern-Simons  theory. 
\begin{enumerate}
\item 
For $G$ any Lie group, a \emph{Chern-Simons theory with gauge group $G$} is an  invariant polynomial $P$ on the Lie algebra of $G$ of degree two and  a multiplicative bundle gerbe $(\mathcal{G},\mathcal{M},\alpha)$ with connection over $G$ whose curvature forms are
\begin{equation}
\label{54}
H = \frac{1}{6} P(  \theta \wedge [ \theta \wedge \theta] ) 
\quad\text{ and }\quad
\rho =  \frac{1}{2}P( p_1^{*}\theta \wedge p_2^{*}\bar\theta
)\text{.}
\end{equation}
Two Chern-Simons theories are considered to be equivalent if their polynomials coincide and their gerbes are multiplicatively isomorphic in the sense of Definition \ref{def2}. 

\item
The \emph{fields} are triples $(M,E)$ of a closed, oriented  three-dimensional manifold $M$ and a principal $G$-bundle $E$ over $M$ with connection. 
\item
A Chern-Simons theory assigns to each field $(M,E,A)$ the number
\begin{equation*}
\mathcal{A}(M,E) := \mathrm{Hol}_{\mathbb{CS}_E(\mathcal{G},\mathcal{M},\alpha)}(M) \in \ueins\text{,}
\end{equation*}
where $\mathbb{CS}_E(\mathcal{G},\mathcal{M},\alpha)$ is the bundle 2-gerbe from Theorem \ref{th2}.

\end{enumerate}
One consequence of this definition is a precise classification of Chern-Simons theories with gauge group $G$. We obtain as a consequence of Proposition \ref{prop3}

\begin{proposition}
\label{prop10}
Let $G$ be an arbitrary Lie group and $P$ be an invariant polynomial on $\mathfrak{g}$ of degree two. 
\begin{enumerate}
\item 
There exist Chern-Simons theories with polynomial $P$ if and only if the pair $(H,\rho) \in M_{\R}^3(G)$ defined by $P$ lies in the integral lattice $M_{\Z}^3(G)$.

\item
If so, inequivalent Chern-Simons theories with polynomial $P$ are   parameterized by $H^3(BG,\ueins)$. 
\end{enumerate}
\end{proposition}

Additionally, every Chern-Simons theory defines  a class in $H^4(BG,\Z)$, namely the multiplicative class of its multiplicative bundle gerbe with connection. For $G$ compact, the action of  $\Omega^2(G)$ on multiplicative bundle gerbes with connection from Proposition \ref{prop1} preserves this class, and if we restrict this action to $\Omega^2_{\mathrm{d},\Delta}(G)$, it also preserves  the curvature forms $\eta$ and $\rho$. Hence, Chern-Simons theories with fixed class in $H^4(BG,\Z)$ are parameterized by 
$\Omega^2_{\mathrm{d},\Delta}(G)/\Omega^2_{\mathrm{d},\Delta,\Z}(G)$, and one can  check that there is a bijection
 \begin{equation*}
\frac{\Omega^2_{\mathrm{d},\Delta}(G)}{\Omega^2_{\mathrm{d},\Delta,\Z}(G)} \cong \frac{H^3(BG,\R)}{H^3(BG,\Z)}\text{.}
\end{equation*}
Since $G$ is compact we have $H^3(BG,\R)=0$ so that  there is no ambiguity; this reproduces a central result of Dijkgraaf and Witten \cite{dijkgraaf1}. 

If $G$ is additionally simple one can also introduce  a \emph{level}: this is the ratio between the pullback of $H$ to the simply-connected cover $\tilde G$, and the generator of $H^3(\tilde G,\Z)=\Z$. For example, we have already found  (see Example \ref{ex3}) that for $G=\so{3}$ the level of a Chern-Simons theory is divisible by four, which is in agreement with the results of \cite{felder2}.

\subsection{Symmetric D-Branes and Bi-Branes}

Every bundle gerbe with connection has a notion of holonomy around closed, oriented surfaces. This notion can be extended to surfaces with boundary by requiring  additional structure, called D-branes. A \emph{D-brane} is a pair $(Q,\mathcal{E})$ of a submanifold $Q$  of $M$ and  a $\mathcal{G}|_Q$-module:  a (not necessarily invertible) 1-morphism $\mathcal{E}: \mathcal{G}|_Q \to \mathcal{I}_{\omega}$, see \cite{carey2, gawedzki4,waldorf1}. 

In a similar way, surface holonomy can be extended to surfaces with \emph{defect lines}: these are embedded oriented circles that divide the surface into several components. To each of these components  a manifold $M_i$ with a bundle gerbe $\mathcal{G}_i$ with connection is assigned.
In this situation, the additional structure is a collection of bi-branes, one for each defect line \cite{fuchs4}. If a defect line separates components assigned to $M_1$ and $M_2$, a \emph{bi-brane} is a submanifold $\tilde Q \subset M_1 \times M_2$ together with a  $p_1^{*}\mathcal{G}_1|_{\tilde Q}$-$p_2^{*}\mathcal{G}_2|_{\tilde Q}$-bimodule, where $p_i$  are the projections $M_1 \times M_2 \to M_i$. As mentioned in Section \ref{sec2}, the bimodule is a 1-morphism
\begin{equation*}
\mathcal{D}: p_1^{*}\mathcal{G}_1|_{\tilde Q} \to p_2^{*}\mathcal{G}_2|_{\tilde Q} \otimes \mathcal{I}_{\tilde\omega}\text{,} 
\end{equation*}
and the 2-form $\tilde \omega\in\Omega^2(\tilde Q)$ is called its curvature.

The goal of this section is the construction of bi-branes in $G \times G$ from given D-branes in $G$.  We consider the twisted multiplication $\tilde m(g,h):=gh^{-1}$ and the  map
\begin{equation*}
\mu: G \times G \to G \times G : (g,h) \mapsto (\tilde m(g,h),h)\text{,}
\end{equation*}
which satisfy the relations $m \circ \mu = p_1$, $p_1 \circ \mu = \tilde m$ and $p_2 \circ \mu = p_2$. 

\begin{definition}
\label{def4}
Let $(\mathcal{G},\mathcal{M},\alpha)$ be a multiplicative bundle gerbe with connection over $G$, and let $(Q,\mathcal{E})$ be a D-brane consisting of a submanifold $Q\subset M$ and a module $\mathcal{E}: \mathcal{G}|_Q \to \mathcal{I}_{\omega}$.
We define a bi-brane $\mathcal{D}_{\mathcal{M}}(Q,\mathcal{E})$ with the submanifold \begin{equation*}
\tilde Q := \tilde m^{-1}(Q) \subset G \times G
\end{equation*}
and the bimodule 
\begin{equation*}
\alxydim{@C=1.8cm}{p_1^{*}\mathcal{G}|_{\tilde Q} \ar[r]^-{\mu^{*}\mathcal{M}^{-1}} & \tilde m^{*} (\mathcal{G}|_{Q}) \otimes p_2^{*}\mathcal{G}|_{\tilde Q} \otimes \mathcal{I}_{-\mu^{*}\rho|_{\tilde Q}} \ar[r]^-{\tilde m^{*}\mathcal{E} \otimes \id \otimes \id} & p_2^{*}\mathcal{G}|_{\tilde Q} \otimes \mathcal{I}_{\tilde\omega}}
\end{equation*}
of curvature $\tilde\omega :=- \mu^{*}\rho|_{\tilde Q} + \tilde m^{*}\omega$.
\end{definition}

In their applications to Wess-Zumino-Witten models, D-branes and bi-branes are required to satisfy certain symmetry conditions. For a \emph{symmetric} D-brane $(Q,\mathcal{E})$ the submanifold $Q$ is  a  conjugacy class $Q=\mathcal{C}_h$ of $G$, and   the gerbe module $\mathcal{E}$ has the particular curvature
\begin{equation}
\label{1}
\omega_h :=  \left \langle \theta|_{\mathcal{C}_h} \wedge \frac{\mathrm{Ad}^{-1}+1}{\mathrm{Ad}^{-1}-1} \;\theta|_{\mathcal{C}_h}  \right \rangle \in \Omega^2(\mathcal{C}_h)\text{.}
\end{equation}
For compact simple Lie groups $G$, all (irreducible)  symmetric D-branes are known: the conjugacy classes are \quot{quantized} to those who correspond to integrable highest weights,  and the gerbe modules $\mathcal{E}$ with curvatures $\omega_h$ have been constructed explicitly  \cite{gawedzki4}. Our aim is to use these available D-branes to construct symmetric bi-branes via Definition \ref{def4}.

In \cite{fuchs4} we have found  conditions for \emph{symmetric bi-branes} in $G \times G$. From  observations of scattering bulk  fields we have deduced the condition that the submanifold $\tilde Q \subset G \times G$ has to be a \emph{biconjugacy class}. These are  the submanifolds
\begin{equation*}
\mathcal{B}_{h_1,h_2}:= \lbrace (g_1,g_2) \in G \times G \;|\; g_1=x_1h_1x_2^{-1}\text{ and }g_2=x_1h_2x_2^{-1}\;\text{;}\;x_1,x_2\in G \rbrace
\end{equation*}
of $G \times G$. Biconjugacy classes are related to conjugacy classes by \begin{equation}
\label{35}
\mathcal{B}_{h_1,h_2} = \tilde m^{-1}(\mathcal{C}_{h_1h_2^{-1}})
\end{equation}
for $\tilde m$ the twisted multiplication used above. 
Another condition we have found in \cite{fuchs4} is that the curvature of a symmetric bi-brane with submanifold $\mathcal{B}_{h_1,h_2}$ has to be the  2-form
\begin{equation*}
\tilde\omega_{h_1,h_2} := \tilde m^{*}\omega_{h_1h_2^{-1}}-\frac{1}{2}\left \langle p_1^{*}\theta \wedge p_2^{*}\theta  \right \rangle \in \Omega^2(\mathcal{B}_{h_1,h_2})
\end{equation*}
with $\omega_{h_1h_2^{-1}}$ the 2-form on $\mathcal{C}_{h_1h_2^{-1}}$ from (\ref{1}).

\begin{proposition}
\label{prop5}
Let $(\mathcal{G},\mathcal{M},\alpha)$ be a multiplicative bundle gerbe with connection over $G$ with curvature $\eta$ (\ref{21}) and 2-form $\rho$ (\ref{5}), and let $(\mathcal{C}_h,\mathcal{E})$ be a symmetric D-brane. Then, the bi-brane $\mathcal{D}_{\mathcal{M}}(\mathcal{C}_h,\mathcal{E})$ is symmetric.
\end{proposition}

\begin{proof}
By (\ref{35}), the submanifold of $\mathcal{D}_{\mathcal{M}}(\mathcal{C}_h,\mathcal{E})$ is a biconjugacy class. Its curvature $\tilde\omega$ is according to Definition \ref{def4} given by $\tilde\omega =- \mu^{*}\rho+ \tilde m^{*}\omega_h$. For $h_1,h_2\in G$ such that $h_1h_2^{-1}=h$, the coincidence $\tilde\omega = \tilde\omega_{h_1,h_2}$ can be checked explicitly. 
\end{proof}

In conclusion, we have constructed first examples of symmetric bi-branes in  Wess-Zumino-Witten models, whose target space is 
 a compact and simple  Lie group $G$.

\begin{remark}
Topological defects in a conformal field theory  have a natural fusion product. It is to expect that symmetric bi-branes also come with a notion of fusion -- some aspects have been developed in \cite{fuchs4}. The ring of topological defects is in turn closely related to the Verlinde ring, and there are concrete manifestations of this relation in terms of bi-branes \cite{fuchs4}. 
Proposition \ref{prop5} allows now to apply the fusion of symmetric bi-branes to symmetric D-branes, whenever the underlying bundle gerbe is multiplicative. 
The observation that symmetric D-branes only have a ring structure in this multiplicative situation is indeed well-known, both in a setup with bundle gerbes   \cite{carey6} as well as in twisted K-theory  \cite{freed1}.
\end{remark}

\section*{Figures}
\addcontentsline{toc}{section}{Figures}
{\tiny
\begin{equation*}
\xymatrix@C=-1.1cm@R=2.7cm{&& \mathcal{M}_{123,4} \circ (\mathcal{M}_{12,3} \otimes \id_{\mathcal{G}_4}) \circ (\mathcal{M}_{1,2} \otimes \id_{\mathcal{G}_3} \otimes \id_{\mathcal{G}_4}) \ar@{=>}[drr]^{\alpha_{1,2,34} \circ \id} \ar@{=>}[dll]_{\id \circ (\alpha_{1,2,3} \otimes \id)} && \\ \mathcal{M}_{123,4} \circ (\mathcal{M}_{1,23} \otimes \id_{\mathcal{G}_4}) \circ (\id_{\mathcal{G}_1} \otimes \mathcal{M}_{2,3} \otimes \id_{\mathcal{G}_4}) \ar@{=>}[dr]_{\alpha_{1,23,4} \circ \id} &&&& \mathcal{M}_{12,34} \circ (\mathcal{M}_{12} \otimes \mathcal{M}_{23}) \ar@{=>}[dl]^{\alpha_{12,3,4} \circ \id} \\ \mathcal{M}_{1,234} \circ (\id_{\mathcal{G}_1} \otimes \mathcal{M}_{23,4}) \circ (\id_{\mathcal{G}_1} \otimes \mathcal{M}_{2,3} \otimes \id_{\mathcal{G}_4}) \ar@{=>}[rrrr]_{\id \circ (\id\otimes\alpha_{1,2,3})}&&&& \mathcal{M}_{1,234} \circ (\id_{\mathcal{G}_1} \otimes \mathcal{M}_{2,34}) \circ \id_{\mathcal{G}_1} \otimes \id_{\mathcal{G}_2 \otimes \mathcal{M}_{3,4}}} \end{equation*}
\begin{figure}[h]
\begin{center}
\caption{The Pentagon axiom for the bimodule morphism of a multiplicative bundle gerbe with connection (Definition \ref{def1}).}
\label{fig1}
\end{center}
\end{figure}
}
{
\tiny
\begin{equation*}
\xymatrix@R=1.2cm@C=2cm{\mathcal{A}_{123} \circ \mathcal{M}_{12,3} \circ (\mathcal{M}_{1,2} \otimes \id_{\mathcal{G}_3}) \ar@{=>}[r]^-{\id \circ \alpha} \ar@{=>}[d]_{\beta_{12,3} \circ \id} & \mathcal{A}_{123} \circ \mathcal{M}_{1,23} \circ (\id_{\mathcal{G}_1} \otimes \mathcal{M}_{2,3}) \ar@{=>}[d]^{\beta_{1,23} \circ \id} \\ \mathcal{M}'_{12,3} \circ (\mathcal{A}_{12} \otimes \mathcal{A}_3) \circ (\mathcal{M}_{1,2} \otimes \id_{\mathcal{G}_3})  \ar@{=>}[d]_{(\beta_{1,2} \otimes \id) \circ \id} & \mathcal{M}'_{1,23} \circ (\mathcal{A}_1 \otimes \mathcal{A}_{23}) \circ (\id_{\mathcal{G}_1} \otimes \mathcal{M}_{2,3}) \ar@{=>}[d]^{\id \circ \beta_{2,3}} \\ \mathcal{M}'_{12,3} \circ (\mathcal{M}_{1,2}' \otimes \id) \circ (\mathcal{A}_1 \otimes \mathcal{A}_2 \otimes \mathcal{A}_3) \ar@{=>}[r]_-{\alpha' \circ \id} & \mathcal{M}'_{1,23} \circ (\id_{\mathcal{G}_1'} \otimes \mathcal{M}'_{2,3}) \circ (\mathcal{A}_1 \otimes \mathcal{A}_2 \otimes \mathcal{A}_3)}
\end{equation*}
\begin{figure}[h]
\begin{center}
\caption{The compatibility axiom between the bimodule morphisms $\alpha$ and $\alpha'$ and the 2-isomorphism $\beta$ of a multiplicative 1-morphism (Definition \ref{def2}). }
\label{fig2}
\end{center}
\end{figure}
}

\kobib{../../tex}

\end{document}